\documentclass{amsart}
\usepackage{amsmath,amssymb,latexsym,amsxtra,amscd}

\input xy
\xyoption{all}

\theoremstyle{plain}
\newtheorem{theorem}{Theorem}[section]
\newtheorem{thm}{theorem}[section]
\newtheorem{cor}[theorem]{Corollary}
\newtheorem{lem}[theorem]{Lemma}
\newtheorem{pro}[theorem]{Proposition}
\newtheorem{conjecture}[theorem]{Conjecture}
\newtheorem*{thm1}{Theorem 1}
\newtheorem*{thm2}{Theorem 2}
\newtheorem*{conj}{Conjecture}

\theoremstyle{definition}
\newtheorem{dfn}[theorem]{Definition}

\newtheorem{rem}[theorem]{Remark}
\newtheorem{block}[theorem]{}

\theoremstyle{remark}
\newtheorem*{notation}{Notation}

\numberwithin{equation}{section}

\newcommand{\lto}{\longrightarrow}
\newcommand{\eto}{\hookrightarrow} 

\newcommand{\ra}{\rightarrow}

\begin{document}

\title[Prym-Tyurin Varieties]{Polarization Types of Isogenous  Prym-Tyurin 
Varieties}

\author{V. Kanev}

\address{V. Kanev\\Dipartamento di Matematica ed Applicazioni\\
         Universit\`{a} di Palermo\\Via Archirafi n.34\\
          90123 Palermo \\Italy \\ and Institute of Mathematics of the Bulgarian Academy of Sciences}
         
\email{kanev@math.unipa.it}

\thanks{The first author was supported in part by the M.I.U.R. project "Geometria delle variet\`{a} algebriche e dei loro spazi di moduli" and 
by Grant MI-1503/2005 of the Bulgarian NSF}

\author{H. Lange}

\address{H. Lange\\Mathematisches Institut\\
              Universit\"at Erlangen-N\"urnberg\\
              Bismarckstra\ss e $1\frac{ 1}{2}$\\
              D-$91054$ Erlangen\\
              Germany}
\email{lange@mi.uni-erlangen.de}
\thanks{}

\keywords{Prym varieties, Prym-Tyurin varieties, polarization type, isogeny}
\subjclass[2000]{Primary 14H40; Secondary 14H30, 14K02}   
\date{}   
                                
\begin{abstract}
Let $p:C\overset{\pi}{\lto}C'\overset{g}{\lto}Y$ be a covering of smooth, projective curves with $\deg(\pi)=2$ and $\deg(g)=n$. Let 
$f:X\to Y$ be the covering of  degree $2^n$, where the curve $X$ parametrizes the liftings in $C^{(n)}$ of
the fibers of $g:C'\to Y$. Let $P(X,\delta)$ be the associated Prym-Tyurin variety, known to be isogenous to the Prym variety
$P(C,C')$.   Most
of the results in the paper focus on calculating the polarization type of the restriction of the canonical polarization
$\Theta_{JX}$ on $P(X,\delta)$.
We obtain the polarization type when $n=3$. When $Y\cong \mathbb{P}^1$ we conjecture that $P(X,\delta)$ is isomorphic 
to the dual of the Prym variety  $P(C,C')$. This was known when $n=2$, we prove it when $n=3$, and for arbitrary $n$ if $\pi:C\to C'$ is 
\'{e}tale. Similar results are obtained for some other types of coverings.
\end{abstract}

\maketitle

\section*{Introduction}

Let $W$ be a finite group which acts on a lattice $L$. Assume the 
representation of $W$ on $L_{\mathbb{C}}=L\otimes_{\mathbb{Z}}\mathbb{C}$ is 
irreducible. Let 
$\lambda\in L_{\mathbb{Q}}$ be a nonzero weight, i.e. 
$w\lambda-\lambda\in L$ for every $w\in W$. Suppose $f:X\to Y$ 
 is a covering of smooth, projective curves, with 
discriminant locus $\mathfrak{D}$, whose monodromy map can be 
decomposed as 
\begin{equation}\label{ei.1}
\pi_1(Y\setminus \mathfrak{D},y_0) \overset{m}{\lto} W \lto S(W\lambda)
\end{equation} 
As shown in \cite{K} one can construct a correspondence in 
$Div(X\times X)$ which induces on 
\[
A=Ker( Nm_f:JX\to JY)^{0}
\]
an endomorphism $\delta$ satisfying a 
quadratic equation $(\delta-1)(\delta+q-1)=0$. The construction in 
\cite{K} is done 
in the case $Y=\mathbb{P}^1$ and its extention to 
the general 
case $g(Y)\geq 0$ is immediate. We outline the necessary changes in 
Section~\ref{s0}.
We let 
\[
P(X,\delta) = (1-\delta)A
\] 
and call this Abelian  variety the {\it Prym-Tyurin variety} associated with 
the covering  $f:X\to Y$. 

Choosing another weight $\lambda'$, one has 
another permutation representation \linebreak
$W\to S(W\lambda')$. Composing with 
$m:\pi_1(Y\setminus \mathfrak{D},y_0) \to W$ and applying Riemann's Existence 
Theorem one obtains another covering $p:C\to Y$,  a correspondence in  
$Div(C\times C)$,
 an induced 
endomorphism $i$ and a Prym-Tyurin variety $P(C,i)$. It is 
known that $P(X,\delta)$ and $P(C,i)$ are isogenous. This can be 
proved by the same arguments as in the case $Y=\mathbb{P}^1$ (see 
\cite{K}, Section~6),  we give an outline in 
Section~\ref{s0}, or using generalized Prym varieties associated to 
Galois coverings (see \cite{D}). 

Let $\overline{\Theta}_{JX}$ and 
$\overline{\Theta} _{JC}$ be the restrictions of the canonical 
polarizations of $JX$ and $JC$ to $P(X,\delta)$ and  $P(C,i)$ 
respectively. A natural question is: what is the connection between the 
polarization types of $\overline{\Theta}_{JX}$ and $\overline{\Theta} 
_{JC}$. Are there simple formulas by which,
if one knows the polarization type associated with a given weight, 
 one can calculate 
the polarization type associated with every other weight?

  In all cases studied in the paper $W$ is a Weyl group of an irreducible root 
system $R$. Here the weights are the usual weights associated 
with irrreducible representations of the corresponding complex simple 
Lie algebras. 

The relation between $P(C,i)$ and $P(X,\delta)$ has been so far studied 
mainly when the restricted polarizations $\overline{\Theta} _{JC}$ and $\overline{\Theta}_{JX}$
are multiples of principal polarizations. In this case usually 
$P(C,i)$ and $P(X,\delta)$ are isomorphic as principally polarized Abelian varieties. For instance
Recillas' construction \cite{Re} corresponds to $Y=\mathbb{P}^1$, $W=W(R)$ where $R$ is a root system of type $A_3$,
$\lambda'=\omega_1$, $\lambda=\omega_2$. Here $\omega_1, \omega_2$ are the fundamental weights of the root system $R$
according to the enumeration in \cite{B}. Similarly Donagi's tetragonal construction corresponds to $Y=\mathbb{P}^1$,\;
$W=W(D_4), \;\lambda'=\omega_1$,\; $\lambda=\omega_3$ or $\lambda=\omega_4$. These two constructions were generalized 
in \cite{K}, when $Y=\mathbb{P}^1$, to $W=W(A_n)$, $\lambda'=\omega_1$, $\lambda=\omega_k, 2\leq k\leq n-1$ and $W=W(D_n)$,
$\lambda'=\omega_1$, $\lambda=\omega_{n-1}$ or $\lambda=\omega_{n}$. In the first case one obtains Prym-Tyurin varieties
isomorphic to Jacobians, in the second one Prym-Tyurin varieties isomorphic to Prym varieties.

The paper \cite{Pa} deals with non-principally polarized Prym varieties. Starting from a double ramified covering 
$\pi:C\to C'$ of a hyperelliptic curve 
$C'$, Pantazis constructs another double ramified covering $\pi':X\to X'$ of a hyperelliptic curve $X'$ and 
proves that the Prym variety $P(X,X')$ is isomorphic to the
dual $\hat{P}(C,C')$. In our set-up this corresponds to $Y=\mathbb{P}^1,\; W=W(B_2), \;\lambda'=\omega_1,\; \lambda=\omega_2$.\\ 

Most of  our results focus on the case $W=W(B_n),\; \lambda'=\omega_1, \;\lambda=\omega_n$. The fundamental weight $\omega_n$ is the
dominant weight of the spinor representation of the Lie algebra $so(2n+1)$, so we call it the spinor weight. Here the relation between
the coverings $p:C\to Y$ and $f:X\to Y$ may be described geometrically by the $n$-gonal construction \cite{D2}. Start  with a 
covering $p:C\overset{\pi}{\lto}C'\overset{g}{\lto}Y$, where $\deg(\pi)=2,\; \deg(g)=n$. Consider the embedding $g^*:Y\to C'^{(n)}$. Let
$Z$ be the curve  defined by the cartesian diagram
\begin{equation*}
\begin{CD}
Z        @>>>         C^{(n)}\\
@VVV                   @VV\pi^{(n)}V\\
Y        @>g^*>>       C'^{(n)}
\end{CD}
\end{equation*}
The points of $Z$ parametrize the liftings of points of $g^*(Y)$ in $C^{(n)}$. Then $X$ is the desingularization of $Z$ and $f:X\to Y$ is the associated covering of degree $2^n$. The two isogenous Prym-Tyurin varieties here are the Prym variety $P(C,C')$ and $P(X,\delta)$, where
 $\delta:JX\to JX$ satisfies the equation $(\delta|_A-1)(\delta|_A+2^{n-1}-1)=0$ when restricted on $A=Ker(Nm_f:JX\to JY)^0$. We prove the following theorem in Section~4.

\begin{thm1}
Let $p:C\overset{\pi}{\lto}C'\overset{g}{\lto}\mathbb{P}^1$ be a covering with $\deg(\pi)=2, \deg(g)=n$.
Let  $P=P(X,\delta)$, $P'=P(C,C')$ and  
let $E_{P}, E_{P'}$ be the Riemann forms of the polarizations $\Theta_{JX}|_P$, $\Theta_{JC}|_{P'}$ respectively.
Let $(d_1,\ldots,d_p)$ be the polarization type of $\Theta_{JC}|_{P'}$. There exists a canonical 
isogeny $\mu: \hat{P}'\to P$ such that
\[
\mu^{*}E_{P}=\frac{2^{n-1}}{d_{1}d_{p}}\hat{E}_{P'}
\]
where $\hat{P}', \hat{E}_{P'}$ are respectively the dual of $P(C,C')$ and the Riemann form  of the dual polarization.
\end{thm1} 
Every covering $p:C\overset{\pi}{\lto}C'\overset{g}{\lto}Y$ with $\deg(\pi)=2$ and $\deg(g)=n$ has monodromy group
contained in $W(B_n)\subset S_{2n}$. We say it  
is a simply ramified covering of type $B_n$ if all local monodromies  are reflections in $W(B_n)$. The branch locus of such
a covering is decomposed as $\mathfrak{D} = \mathfrak{D}_s\cup \mathfrak{D}_{\ell}$ according to the type of reflections: 
with respect to short roots; or with respect to long roots. One has that $g:C'\to Y$ is simply ramified in $\mathfrak{D}_{\ell}$,
$\pi:C\to C'$ is branched in $|\mathfrak{D}_s|$ points and $\pi(Discr(C\to C'))=\mathfrak{D}_s$. We conjecture the following statement
holds.

\begin{conj}
Assume $p:C\overset{\pi}{\lto}C'\overset{g}{\lto}\mathbb{P}^1$ is a simply 
ramified covering of type $B_n$.   Then the isogeny 
$\mu: \hat{P}(C,C')\to P(X,\delta)$ is an isomorphism. 
\end{conj}
A positive answer to the conjecture would give the polarization type of \linebreak
$\Theta_{JX}|_{P(X,\delta)}$. When $n=2$ the statement of the conjecture is known to be true and is due to Pantasis \cite{Pa}. We give further evidence verifying
it when $n=3$ (see Theorem~\ref{s1.35}), as well as when $n$ is arbitrary and $\pi:C\to C'$ is \'{e}tale (Proposition~\ref{s10.26}). 
When $W=W(B_3)$,  $\lambda'=\omega_1$, $\lambda=\omega_3$, we obtain the polarization type of $\Theta_{JX}|_{P(X,\delta)}$ proving in 
Section~\ref{s1} the following result.

\begin{thm2}
Let $p:C\overset{\pi}{\lto}C'\overset{g}{\lto}Y$, with $\pi$ and $g$ ramified of degrees 2 and 3 respectively, 
be a simply ramified 
covering of type $B_3$. 
Let $f:X\to Y$ be the 
covering of degree 8 obtained by the 3-gonal construction and let 
$P(X,\delta)$ be the Prym-Tyurin variety. Then the polarization types 
of $\Theta_{JC}|_{P(C,C')}$ and
 $\Theta_{JX}|_{P(X,\delta)}$ are respectively
\begin{align*}
&(\underbrace{1,\ldots,1}_{\frac{1}{2}|\mathfrak{D}_{s}|-1},
\underbrace{2,\ldots,2}_{\frac{1}{2}|\mathfrak{D}_{\ell}|+2g(Y)-2},
\underbrace{2,\ldots,2}_{g(Y)})\\
&(\underbrace{2,\ldots,2}_{\frac{1}{2}|\mathfrak{D}_{\ell}|+2g(Y)-2},
\underbrace{4,\ldots,4}_{\frac{1}{2}|\mathfrak{D}_{s}|-1}, 
\underbrace{8,\ldots,8}_{g(Y)})
\end{align*}
If $Y\cong \mathbb{P}^1$, then  
$P(X,\delta)$ is isomorphic to the dual of $P(C,C')$.
\end{thm2}
Assuming $C'\cong \mathbb{P}^1$ in Theorem~2 we obtain a new family of principally polarized Prym-Tyurin varieties.
They have exponent 4 and are isomorphic to hyperelliptic Jacobians (Corollary~\ref{s1.39}). \\

The contents of the paper are as follows. Section~\ref{s0} contains preliminary material. In Section~\ref{s10} 
we study coverings $p:C\overset{\pi}{\lto}C'\overset{g}{\lto}Y$ with $\deg(\pi)=2, \deg(g)=n$ with the emphasis
on their connection with root systems of type $B_n$. We calculate the possible monodromy groups of the simply ramified coverings
of type $B_n$ under some restrictions. Applying the general constructions of Section~\ref{s0} we associate to a given covering
$p:C\overset{\pi}{\lto}C'\overset{g}{\lto}Y$ of
type $B_n$ and the spinor weight a covering $f:X\to Y$ of degree $2^n$, define and study  canonical correspondences in $Div(X\times C)$ and 
identify $X$ with the $n$-gonal construction \cite{D2}. Section~\ref{s2} contains some results about $P(X,\delta)$ valid for every $n$. 
In Section~\ref{s0a} we recall the definition of dual Abelian variety and dual polarization, prove a criterion which relates the Riemann forms of isogenous Prym-Tyurin varieties, give a proof of Theorem~1 and verify the isomorphism $\mu: \hat{P}(C,C')\overset{\sim}{\lto} P(X,\delta)$ when 
$\pi:C\to C'$ is \'{e}tale. In Section~\ref{s1} we study coverings associated with Weyl groups of root systems of rank $\leq 3$. We give a new and simpler proof of Pantazis' result ($W=W(B_2)$). We briefly discuss the simple case $W=W(A_3)$, the Recillas' construction. We give a proof of 
Theorem~2, and finally make explicit calculation of the Prym-Tyurin varieties and the polarization types when the 
monodromy group is contained in $W(D_3)$. In Section~\ref{s6} we study the cases $W=W(B_4)$ and $W=W(D_4)$. We cannot say much about the polarization types, however we include some relation between the Abelian varieties involved.

\medskip

\begin{notation}
Let $X=V/\Lambda$ and $X'=V'/\Lambda'$ be complex tori. If $f:X\to X'$ is a homomorphism, we denote by $\tilde{f}:V\to V'$ the unique $\mathbb{C}$-linear map with $\tilde{f}(\Lambda)\subset \Lambda'$ inducing $f$. Identifying $V$ and $V'$ with the tangent spaces $T_0X$ and $T_0X'$ we have that $\tilde{f}$ equals the differential $df$ at 0. Furthermore, identifying
$\Lambda$ and $\Lambda'$ with the homology groups $H_1(X,\mathbb{Z})$ and $H_1(X',\mathbb{Z})$ respectively, one has that 
$\tilde{f}|_\Lambda$ equals $f_{*}$, the induced homomorphism on homology. Similarly $\tilde{f}:V\to V'$ and 
$f_{*}:H_1(X,\mathbb{R})\to H_1(X',\mathbb{R})$ are the same maps if we make the identification of $V$ and $V'$ with $H_1(X,\mathbb{R})$ and $H_1(X',\mathbb{R})$ respectively as real vector spaces.
\end{notation}

\section{Prym-Tyurin varieties}\label{s0}
\begin{block}\label{s0.10}
Let $\mathbb{R}^{d}=\oplus_{i=1}^{d}\mathbb{R}e_{i}$, 
$\mathbb{R}^{e}=\oplus_{j=1}^{e}\mathbb{R}f_{j}$. Let $W$ be a finite 
group and let $W \to GL(\mathbb{R}^{d})$ and $W\to GL(\mathbb{R}^{e})$ 
be linear permutation representations, such that $W$ acts transitively 
on both $\{e_{1},\ldots,e_{d}\}$ and $\{f_{1},\ldots,f_{e}\}$. Let 
$A=(a_{ij})$ be a $d\times e$ matrix with integer entries and let 
$S:\mathbb{R}^{d}\to \mathbb{R}^{e}$ be the linear map given by 
$$
S(e_{i})=\sum_{j=1}^{e}a_{ij}f_{j}.
$$ 
Let $T:\mathbb{R}^{d}\to \mathbb{R}^{e}$,
$T_{1}:\mathbb{R}^{d}\to \mathbb{R}^{d}$ and 
$T_{2}:\mathbb{R}^{e}\to \mathbb{R}^{e}$ be linear maps which in the 
bases $\{e_{i}\}$, $\{f_{j}\}$ have matrices with all entries equal to 
1. We denote by 
${^{t}S}:\mathbb{R}^{e}\to \mathbb{R}^{d}$ and 
${^{t}T}:\mathbb{R}^{e}\to \mathbb{R}^{d}$ the linear maps with 
transposed matrices.
\end{block}
\renewcommand{\theenumi}{\roman{enumi}}
\begin{lem}\label{s0.10a}
Suppose $S$ is $W$-equivariant. Then there exist integers $a,b\in 
\mathbb{Z}$ such that 
\begin{enumerate}
\item
${^{t}T}\cdot S = {^{t}S}\cdot T = aT_{1}$,
\item
$T\cdot {^{t}S} = S\cdot {^{t}T} = bT_{2}$.
\end{enumerate}
\end{lem}
\begin{proof}
\[
S(\sum_{i=1}^{d}e_{i}) = \sum_{j=1}^{e}(\sum_{i=1}^{d}a_{ij})f_{j} := 
\sum_{j=1}^{e}b_{j}f_{j}
\]
Let $w\in W$. One has
\[
S(\sum_{i=1}^{d}e_{i}) = S(\sum_{i=1}^{d}w(e_{i})) = 
wS(\sum_{i=1}^{d}e_{i}) = \sum_{j=1}^{e}b_{j}(w(f_{j}))
\]
Since $W$ acts transitively on $\{f_{j}\}$ we conclude that 
$b_{1}=\cdots =b_{e}=b$. So, $S\cdot {^{t}T}=bT_{2}$. Transposing we 
obtain $T\cdot{^{t}S}=bT_{2}$. This proves (ii). The proof of (i) is similar 
considering ${^{t}S}(f_{j})=\sum_{i=1}^{d}a_{ij}e_{i}$. One has 
$a=\sum_{j=1}^{e}a_{ij}$ for every i.
\end{proof}
\begin{block}\label{s0.11}
Let $W$, $\{e_{i}\}$, $\{f_{j}\}$ and $S$ be as in \S\ref{s0.10}. 
Suppose $C,X$ and $Y$ are smooth, projective curves, $Y$ is 
irreducible (but $C$ and $X$ might be reducible). Let 
$$
f:X\to Y \quad  \mbox{and} \quad  p:C\to Y
$$
 be coverings of degrees $d$ and $e$ respectively which are 
not branched in $Y\setminus \mathfrak{D}$. Let $y_{0}\in 
Y\setminus \mathfrak{D}$. Suppose the monodromy maps of $f$ and $p$ 
can be decomposed as 
$$
\pi_1(Y\setminus \mathfrak{D},y_0) \overset{m}{\to}W\to S_{d}
\quad \mbox{and} \quad \pi_1(Y\setminus \mathfrak{D},y_0) 
\overset{m}{\to}W\to S_{e}, 
$$
where $W$ acts transitively on 
$\{e_{1},\ldots,e_{d}\}$ and $\{f_{1},\ldots,f_{e}\}$ as in 
\S\ref{s0.10}. 

A $W$-equivariant linear map $S$ as in \S\ref{s0.10} 
induces a correspondence, which abusing notation,  we denote again by 
$S$. It is defined as follows. Fix  bijections $f^{-1}(y_0)\to 
\{e_{1},\ldots,e_{d}\}$ and 
$p^{-1}(y_0)\to \{f_{1},\ldots,f_{e}\}$. If $y \in Y\setminus \mathfrak{D}$, 
choose a path $\gamma$ in $Y\setminus \mathfrak{D}$ which connects $y$ with 
$y_{0}$.  Enumerate the points of the 
fibers over $y$ using covering homotopy along $\gamma$, so 
$f^{-1}(y)=\{x_{1},\ldots,x_{d}\}$ and 
$p^{-1}(y)=\{z_{1},\ldots,z_{e}\}$. Then 
\[
S(x_{i})=\sum_{j=1}^{e}a_{ij}z_{j}, \qquad 
{^{t}S}(z_{j})=\sum_{i=1}^{d}a_{ij}x_{i}
\]
Let $T\in Div(X\times C)$, $T_{1}\in Div(X\times X)$ and $T_{2}\in 
Div(C\times C)$ be the trace correspondences 
$$
T(x)=p^{*}(f(x)), \quad T_{1}(x)=f^{*}(f(x)), \quad T_{2}=p^{*}(p(x)).
$$ 
Lemma~\ref{s0.10a} yields
\begin{equation}\label{es0.12}
\begin{split}
{^{t}T}\circ S &= {^{t}S}\circ T = (\deg S)T_{1}\\
T\circ {^{t}S} &= S\circ {^{t}T} = (\deg {^{t}S})T_{2}
\end{split}
\end{equation}
\end{block}
\begin{lem}\label{s0.12}
Let 
$$
A=Ker (Nm_{f}:JX\to JY)^{0}, \quad B=Ker (Nm_{p}:JC\to JY)^{0}.
$$ 
Then the endomorphism $s:JX\to JC$ induced by the corespondence $S$ 
transforms $A$ into $B$ and $f^{*}(JY)$ into $p^{*}(JY)$.
\end{lem}
\begin{proof}
Let $y\in Y$ and $z\in p^{-1}(y)$. Using \eqref{es0.12} one has 
$S(f^{*}(y)) = S({^{t}T}(z)) = (\deg {^{t}S})T_{2}(z) = (\deg 
{^{t}S})p^{*}(y)$. So $s(f^{*}(JY))\subset p^{*}JY$. If $x\in X$ and 
$y=f(x)$, then $f^{*}Nm_{p}(S(x)) = {^{t}T}(S(x)) = (\deg 
S)f^{*}(f(x))$, so for every $u\in JX$ one has 
$f^{*}Nm_{p}(s(u))=(\deg S)f^{*}Nm_{f}(u)$. This implies $s(A)\subset 
B$ since $f^{*}:JY\to JX$ has finite kernel and $A=Ker (Nm_{f})^{0}$ 
is connected.
\end{proof}
\begin{block}\label{s0.14}
Let $W$, $L$, $\lambda \in L_{\mathbb{Q}}, \lambda \neq 0$, 
$f:X\to Y$ be as in the introduction. We assume $Y$ is irreducible. 
Let $W\lambda =\{\lambda_{1},\ldots,\lambda_{d}\}$. Let $(\;|\;)$ be 
a symmetric, $W$-invariant, negative definite bilinear form such that 
$(w\lambda-\lambda|\lambda)\in \mathbb{Z}$ for $\forall w\in W$. The 
construction in \cite{K}, Section~4 yields a lattice 
$N(R,\lambda)\cong L\oplus \mathbb{Z}$, an action of $W$ on 
$N(R,\lambda)$, a symmetric bilinear $W$-invariant form $(\;,\:)$ 
on $N(R,\lambda)$, which extends $(\;|\;)$, an orbit $W(\ell) = 
\{\ell_{1},\ldots,\ell_{d}\}\subset N(R,\lambda)$ and a 
$W$-equivariant bijection $\ell_{i}\leftrightarrow \lambda_{i}$. 
One has
\begin{equation}\label{es0.14}
(\ell_{i},\ell_{j}) = (\lambda_i|\lambda_j-\lambda_i)-1
\end{equation}
and the following properties hold: $(\ell_{i},\ell_{j})\in \mathbb{Z}$;  
$(\ell_{i},\ell_{i})=-1$ and $(\ell_{i},\ell_{j})\geq 0$ for $i\neq 
j$. Let $E\in GL_d(\mathbb{R})$ be the identity.
One considers the $W$-equivariant linear map 
$G=D-E:\mathbb{R}^{d}\to \mathbb{R}^{d}$ given by
\[
G(e_i)=\sum_{j=1}^{d}(\ell_{i},\ell_{j})e_{j}
\]
One defines a correspondence $D-\Delta\in Div(X\times X)$ as in 
\S\ref{s0.11}. Here $\Delta$ is the diagonal. Let $A=Ker(Nm_{f}:JX\to 
JY)^{0}$. 
\end{block}
\begin{pro}\label{s0.14a}
The endomorphism $\delta :JX\to JX$ induced by $D$ leaves invariant 
$A$ and $f^{*}(JY)$. The restriction $\delta|_{A}:A\to A$ satisfies 
the quadratic equation 
$$
(\delta|_{A}-1)(\delta|_{A}+q-1)=0,
$$
where $q$, the exponent of the correspondence, 
is the integer 
$$
q=-d(\lambda|\lambda)/rk(L).
$$ 
Furthermore $q\geq 1$ and 
$q=1$ if and only if $D=0$.
\end{pro}
\begin{proof}
This follows from Lemma~\ref{s0.12}, the equality $G(G+qE)=mT$ proved 
in \cite{K}, Proposition~5.3 and the argument of \cite{K}, \S5.4.
\end{proof}
Lemma~\ref{s0.12} and Proposition~\ref{s0.14a} show that the  
differential of $\delta$ at 0
\[
d\delta:T_{0}(JX)\to T_{0}(JX)
\] 
can be diagonalized and has 
eigenvalues $\deg D, 1, 1-q$ with eigenspaces which correspond to the 
Abelian subvarieties $f^{*}(JY),(\delta+q-1)A,(1-\delta)A$ respectively.

\begin{dfn}\label{s0.15}
We let $P(X,\delta)=(1-\delta)A$ and call this Abelian variety the 
Prym-Tyurin variety associated with the covering $f:X\to Y$. 
\end{dfn}
\begin{rem}\label{s0.15a}
We notice that the choice of the $W$-invariant bilinear form 
$(\;|\;)$ is irrelevant for the definition of the Prym-Tyurin 
variety. 
Indeed, multiplying $(\;|\;)$ by an integer $k$ one changes $G$ to 
$G'=kG+(k-1)T$, so $1-\delta'|_{A}=k(1-\delta|_{A})$ and 
therefore the Prym-Tyurin variety remains the same.
\end{rem}
\begin{block}\label{s0.16a}
Given the set-up of \S\ref{s0.14}, namely a covering $f:X\to Y$ whose 
monodromy map can be decomposed as 
$\pi_1(Y\setminus \mathfrak{D},y_0) \overset{m}{\to}W\to S(W\lambda)$, 
let us choose another $W$-orbit of 
weights 
$W\lambda'=\{\lambda'_{1},\ldots,\lambda'_{e}\}\subset L_{\mathbb{Q}}$. 
One associates with 
$\pi_1(Y\setminus \mathfrak{D},y_0) \overset{m}{\to}W\to S(W\lambda')$ 
a covering $f':X'\to Y$. Let 
$$
P\subset A \quad  \mbox{and} \quad  P'\subset A'=Ker(Nm_{f'}:JX'\to JY)^{0}
$$ 
be the two Prym-Tyurin varieties.
\end{block}
\pagebreak
\begin{pro}\label{s0.16}
The Abelian varieties $P$ and $P'$ are isogenous.
\end{pro}
\begin{proof}
This is proved when $Y=\mathbb{P}^1$ in \cite{K}, Section~6. The proof 
for arbitrary $Y$ is essentially the same, so we only sketch it, referring 
for details to \cite{K}. One chooses an appropriate $W$-invariant 
bilinear form $(\;|\;)$ on $L$, constructs a lattice $N\supset L$, an 
action of $W$ on $N$, a $W$-invariant, integer-valued, bilinear form 
$B$ on $N$, such that $B|_{L}=(\;|\;)$, and two $W$-orbits 
$\{\ell_1,\ldots,\ell_d\}$ and $\{m_1,\ldots,m_e\}$ in $N$ which are 
$W$-equivariantly bijective to $W\lambda$ and $W\lambda'$ 
respectively. One considers the linear map $S:\mathbb{R}^{d}\to 
\mathbb{R}^{e}$ given by 
\[
S(e_{i}) = \sum_{j=1}^{e}B(\ell_{i},m_{j})f_{j}
\]
This is a $W$-equivariant linear map and one employs it to construct a 
correspondence $S\in Div(X\times X')$ as in \S\ref{s0.11}. By 
Lemma~\ref{s0.12} it induces an endomorphism $s:A\to A'$ and similarly 
${^{t}S}$ induces ${^{t}s}:A'\to A$. By Remark~\ref{s0.15a} we may use 
the chosen common $(\;|\;)$ in order to calculate $P$ and $P'$. Let 
$G, G', \delta, \delta', q$ and $q'$ be the corresponding data as in 
\S\ref{s0.14} and Proposition~\ref{s0.14a}. The following equalities 
of linear maps are verified in the course of the proof of 
Theorem~6.5 of \cite{K}:
\begin{equation}\label{es0.18a}
(G'+q'E)\cdot S = c_{1}T   ,\qquad S\cdot (G+qE) = c_{2}T
\end{equation}
for some $c_1,c_2\in \mathbb{Z}$.
Transposing and passing to correspondences and homomorphisms this 
implies that 
\begin{equation}\label{es0.18b}
\begin{split}
s(A) \subset P', \qquad  & s((\delta+q-1)A) = 0\\
{^{t}s}(A')\subset P,\qquad  &{^{t}s}((\delta'+q'-1)A') = 0
\end{split}
\end{equation}
One has by \cite{K}, Lemma~6.5.1 that 
\begin{equation}\label{es0.18c}
{^{t}S}\cdot S = -q'G + d_1T_1, \qquad S\cdot {^{t}S} = -qG' + 
d_{2}T_{2}
\end{equation}
for some $d_{1},d_{2}\in \mathbb{Z}$. Passing to correspondences and 
induces homomorphisms of $A$ and $A'$ this implies 
\begin{equation}\label{es0.18d}
{^{t}s}\circ s = q'(1-\delta|_{A}), \qquad s\circ {^{t}s} = 
q(1-\delta'|_{A'})
\end{equation}
Restricting to $P$ and $P'$ one obtains 
\begin{equation}\label{es0.18}
{^{t}s}\circ s|_{P}=q q'\cdot id_{P}, \qquad s\circ 
{^{t}s}|_{P'} = q q'\cdot id_{P'}
\end{equation}
So $s|_{P}:P\to P'$ and ${^{t}s}|_{P'}:P'\to P$ are isogenies.
\end{proof}
\section{Coverings of type $B_n$}\label{s10}
\begin{block}\label{s0.1}
Let us consider a real vector space $\mathbb{R}^{n}$ with basis 
$\epsilon_{1},\ldots,\epsilon_{n}$ and a cup product 
$(\epsilon_{j}|\epsilon_{k})=\delta_{jk}$. Denote by $R$ the set 
\[
R = \{\pm\epsilon_{j}|j=1,\ldots,n\} \cup 
\{\pm\epsilon_{j}\pm\epsilon_{k}|1\leq j<k\leq n\} = R_{s}\cup R_{\ell}
\]
This is a root system of type $B_{n}$ with $R_s$ the set of short 
roots and $R_{\ell}$ the set of long roots. For every $\alpha\in R$ let 
$s_{\alpha}:\mathbb{R}^{n}\to \mathbb{R}^{n}$ be the reflection 
$s_{\alpha}(x)=x-\frac{2(x|\alpha)}{(\alpha|\alpha)}\alpha$. Then 
$s_{\alpha}(R)=R$. The finite group $W$ generated  by $s_{\alpha}, 
\alpha\in R$ is the Weyl group of type $B_{n}$. 
The set of long roots $R_{\ell}$ is a root system of type $D_n$ and the
 reflections with respect to long roots 
 generate a subgroup of index 2 in $W$. We will
usually denote the two Weyl groups by $W(B_n)$ and $W(D_n)$.

Consider the orbit 
$W\epsilon_{1}=R_{s}=\{\epsilon_{1},-\epsilon_{1}, \ldots , 
\epsilon_{n},-\epsilon_{n}\}$. Acting on this set, the Weyl group
$W(B_n)$ is 
identified with the permutation subgroup of $S(W\epsilon_{1})$ 
consisting of permutations which commute with $-id$ (cf \cite{B}). 
Each reflection $s_{\epsilon_{j}}$ acts as a transposition: 
$\epsilon_{j}\leftrightarrow 
-\epsilon_{j}$ and each reflection $s_{\epsilon_{j}\pm \epsilon_{k}}$ 
acts as a product of two independant transpositions.\\

Let $W=W(B_n)$. Let $Y$ be a smooth, projective, irreducible curve. Let 
$\mathfrak{D}\subset Y$ be a finite subset and let $y_{0}$ be a point 
in $Y\setminus \mathfrak{D}$. Let $m:\pi_1(Y\setminus 
\mathfrak{D},y_0) \to W$ be a homomorphism. Composing with $W\eto 
S(W\epsilon_{1})=S_{2n} $ and applying Riemann's existence theorem one 
obtains a covering 
\(
p:C\to Y
\) 
of degree $2n$ whose monodromy map
decomposes as 
\begin{equation}\label{e0.1}
\pi_1(Y\setminus \mathfrak{D},y_0) \overset{m}{\lto} W \lto 
S_{2n}
\end{equation} 
It is convenient to denote $\epsilon_{-j}=-\epsilon_{j}$ and consider 
$S_{2n}$ as the permutation group on the elements 
$\{1,-1,\ldots,n,-n\}$. Since the elements of $W$ commute with $-id$, the 
curve $C$ is 
equipped with an involution $i:C\to C$. Let $C'=C/i$ and let $\pi : C\to C'$ 
be the quotient map. One obtains the following decomposition  of $p$ 
\begin{equation} \label{e0.2}
p:C\overset{\pi}{\lto}C'\overset{g}{\lto}Y.
\end{equation}
Conversely, suppose a morphism 
$p:C\to Y$ of smooth, projective curves, can be 
decomposed as $C\overset{\pi}{\lto}C'\overset{g}{\lto}Y$ with $\deg 
\pi =2,\; \deg g =n$. Assume  $Y$ is irreducible. Let $i:C\to C$ be the 
involution with $\pi 
\circ i=\pi$. Let $\mathfrak{D}$ be the discriminant locus of $p$, let 
$y_{0}\in Y\setminus \mathfrak{D}$ and let 
$p^{-1}(y_{0})=\{x_{1},x'_{1},\ldots,x_{n},x'_{n}\}$ where 
$x'_{k}=i(x_{k})$. Denoting $x_{-j}=x'_{j}$ and identifying 
$x_{j}$ with $\epsilon_{j}$, and $x_{-j}$ with 
$-\epsilon_{j}=\epsilon_{-j}$, $j=1,\ldots,n$, we see that the 
monodromy map $\pi_1(Y\setminus \mathfrak{D},y_0) \to S_{2n}$ can be 
decomposed as in \eqref{e0.1}.

We call, for easier reference, a covering 
$p:C\overset{\pi}{\lto}C'\overset{g}{\lto}Y$ with $\deg \pi =2$\; a 
\emph{covering of type $B_n$}, or shortly a \emph{$B_n$-covering}. 
It  has \emph{simple ramification of type $B_n$} at a point $b\in 
\mathfrak{D}$,
   if the local 
monodromy at $b$ is a reflection in $W(B_n)$. If this is
a reflection $s_{\alpha}$ with $\alpha\in R_{\ell}$ (a long root), 
then $g:C'\to Y$ is simply ramified at $b$, i.e. $|g^{-1}(b)|=n-1$, and 
$\pi:C\to C'$ is unramified at $g^{-1}(b)$. If $\alpha \in R_{s}$ (a 
short root), then $g$ is unramified at $b$ and $\pi$ has one branch 
point among $g^{-1}(b)$. We say 
a ramified covering $p:C\to Y$, which can be decomposed as in
\eqref{e0.2}, is a 
\emph{simply ramified $B_n$-covering} if it has simple ramification of type 
$B_n$ at 
all discriminant points. In this case
we denote by $\mathfrak{D}=\mathfrak{D}_{s}\cup 
\mathfrak{D}_{\ell}$ the corresponding splitting.
It is clear that $p:C\overset{\pi}{\lto}C'\overset{g}{\lto}Y$ is a simply 
ramified $B_n$-covering 
 if and only if 
the covering $g:C'\to Y$ is either unramified, or  simply  
ramified,
$g(Discr(C\to C'))\cap Discr(C'\to Y)=\emptyset$,
and no two discriminant points of $\pi$ belong to the same fiber of $g$. 

The ordinary simply ramified coverings
$f:X\to Y$ of degree $n$, where the local monodromies are transpositions, may be considered
as simply ramified coverings of type $A_{n-1}$, since $W(A_{n-1})\cong S_n$. A simply ramified
$B_n$-covering $p:C\overset{\pi}{\lto}C'\overset{g}{\lto}Y$
is  an ordinary simply ramified covering if and only if $\mathfrak{D}_{\ell}=\emptyset$, i.e. 
$g:C'\to Y$ is \'{e}tale. 
\end{block}
\begin{block}\label{s0.2a}
Given $m: \pi_1(Y\setminus \mathfrak{D},y_0) \to W$  and $p:C\to Y$ as 
above, one calculates, using \eqref{es0.14}, that the correspondence 
associated with $L=Q(R), W\epsilon_{1}=\{\epsilon_{\pm 
j}|j=1,\ldots,n\}$ and $(\epsilon_{j}|\epsilon_{k})=-\delta_{jk}$ is 
obtained from the linear map $D(\ell_{j})=\ell_{-j},\: j\in [-n,n]$. Therefore 
this 
correspondence is the graph of the involution $i:C\to C$. Here 
$(1-i)JC\subset A$, where $A=Ker(Nm_{p})^{0}$, so $P(C,i)=(1-i)A$ is 
the ordinary Prym variety $P(C,C')$.
\end{block}
\begin{block}\label{s0.50}
Let $p:C\overset{\pi}{\lto}C'\overset{g}{\lto}Y$ be a covering with $\deg \pi =2$ and 
irreducible $C$. We want to calculate the monodromy group of 
$p:C\to Y$ in  case this is a simply ramified 
$B_n$-covering.
This is  
equivalent to calculating the Galois group of the field extension 
$\mathbb{C}(Y)\subset K$, where $K$ is the Galois hull of $\mathbb{C}(C)$ over $\mathbb{C}(Y)$.

Let us first recall some facts about ordinary coverings. 
Let $f:X\to 
Y$ be a covering of smooth curves of degree $n$. The curve $Y$ is 
assumed irreducible. Let $G\subset S_n$ be the monodromy group. The 
irreducibility of $X$ is equivalent to the transitivity of $G$. Let us 
recall (see \cite{Wi}) that a transitive subgroup of $S_n$ is called 
imprimitive if there is a a subset $\Phi\subset\{1,\ldots,n\}, \quad 
1<|\Phi|<n$, such that for every $g\in G$ one has either 
$g(\Phi)=\Phi$, or $g(\Phi)\cap \Phi=\emptyset$. A transitive 
group is called primitive if it is not imprimitive. It is clear that 
the monodromy group of $f:X\to Y$ is imprimitive if and only if there 
is a nontrivial decomposition $f:X\to X_1\to Y$ . We call a covering 
of smooth, irreducible curves $f:X\to Y$ primitive if no such 
decomposition exists. If $f:X\to Y$ has simple ramification, then for 
any decomposition $X\to X_1\to Y$ with $\deg(X\to X_1)>1$ the covering 
$X_1\to Y$ must be \'{e}tale. So, for a simply ramified covering of 
smooth irreducible curves $f:X\to Y$ primitivity is equivalent to the 
surjectivity of $f_{*}:\pi_{1}(X,*)\to \pi_{1}(Y,*)$. 

It is shown in 
\cite{BE}, Lemma~2.4 that a primitive subgroup of $S_n$ which contains 
a transposition equals $S_n$. So, if $f:X\to Y$ is a primitive 
covering of degree $n$ with at least one simple branching, then the 
monodromy group is $S_n$. 
\end{block}
\begin{block}\label{s0.52}
Using the notation of \S\ref{s0.1} let 
$a_j=\{\epsilon_{j},-\epsilon_{j}\}$ and denote 
$\Sigma=\{a_{1},\ldots,a_{n}\}$. Every element of $W(B_n)$ induces a 
permutation of $\Sigma$. One obtains the following exact sequence.
\begin{equation}\label{es0.52}
0\to G_{2}\to W(B_n)\to S_n\to 0
\end{equation}
where $G_{2}\cong (\mathbb{Z}/2\mathbb{Z})^n$ is the subgroup generated by the 
reflections $s_{\epsilon_{j}}, j=1,\ldots,n$. Let $G_{1}$ be the 
subgroup generated by $s_{\epsilon_{i}-\epsilon_{j}}, 1\leq i<j\leq 
n$. It maps isomorphically to $S_n$, so $W(B_n)$ is a semidirect 
product of $G_{1}$ and $G_{2}$. Furthermore $G_{1}$ has two orbits 
when acting on $\{\pm \epsilon_{j}|j=1,\ldots,n\}$:\quad 
$\Sigma_{1}=\{\epsilon_{1},\ldots,\epsilon_{n}\}$ and 
$\Sigma_{2}=\{-\epsilon_{1},\ldots,-\epsilon_{n}\}$. The conjugates of 
$G_{1}$ map surjectively to $S_{n}$ in \eqref{es0.52}. Two other 
subgroups of $W(B_n)$ which have the same property are: $W(B_n)$ 
itself; the group $W(D_n)$ generated by the reflections with respect 
to long roots $\{s_{\epsilon_{j}}\pm \epsilon_{k}|1\leq j<k\leq n\}$. 
The conjugates of the group of the next lemma is another example.
\end{block}
\begin{lem}\label{s0.52a}
Let $\sigma =-id\in W(B_n)$. Let $G=N_{W(B_n)}(G_{1})$ be the 
normalizer of $G_{1}$. Then $G=G_{1}\cup G_{1}\sigma$. The group $G$ 
acts transitively on $\{\pm \epsilon_{j}|j=1,\ldots,n\}$ and if $n\geq 
3$ the only reflections which belong to $G$ are 
$\{s_{\epsilon_{j}-\epsilon_{k}}|1\leq j<k\leq n\}$.
\end{lem}
\begin{proof}
Let $g\in N_{W(B_n)}(G_1)$. then either 
$g(\Sigma_{1})=\Sigma_{1}, g(\Sigma_{2})=\Sigma_{2}$ or 
$g(\Sigma_{1})=\Sigma_{2},g(\Sigma_{2})=\Sigma_{1}$. In the first case 
$g\in G_{1}$. In the second $g\in G_{1}\sigma$. If $n\geq 3$ no 
reflection can satisfy $s_{\alpha}(\Sigma_{1})=\Sigma_{2}$, hence the 
reflections which belong to $G_{1}$ are 
$s_{\epsilon_{j}-\epsilon_{k}}, 1\leq j<k\leq n$.
\end{proof}
\begin{lem}\label{s0.53}
Let $G\subset W(B_n)$ be a subgroup, which contains a reflection with 
respect to a long root, and the image of $G$ in $S_n$ is a primitive 
group. Then one of the following alternatives holds.
\renewcommand{\theenumi}{\roman{enumi}}
\begin{enumerate}
\item 
$G=W(B_n)$,
\item
$G=W(D_n)$,
\item
$n\geq 3$ and $G=wN_{W(B_n)}(G_1)w^{-1}$ for some $w\in W(B_n)$,
\item
$G=wG_{1}w^{-1}$ for some $w\in W(B_n)$.
\end{enumerate}
Furthermore $G$ is transitive only in cases (i) -- (iii) and in cases 
(iii) and (iv) the set of reflections in $G$ equals 
$w\{s_{\epsilon_{j}-\epsilon_k}|1\leq j<k\leq n\}w^{-1}$.
\end{lem}
\begin{proof}
The image $\overline{G}$ of $G$ in $S_n$ is a primitive subgroup which 
contains a transposition. So, $\overline{G}=S_n$ by \cite{BE}, Lemma~2.4. 
By hypothesis $G$ contains some $s_{\alpha}, \alpha \in R_{\ell}$. Let 
$w_{1}\in W(B_n), w_{1}(\alpha)=\epsilon_{1}-\epsilon_{2}$. Substituting 
$G$ by $w_{1}Gw_{1}^{-1}$ we may assume that 
$s_{\epsilon_{1}-\epsilon_{2}}\in G$. Let $h\in S_n$ be a permutation 
with $h(1)=2,h(2)=3$. Let $g\in G,\; \overline{g}=h$. Then 
$gs_{\epsilon_1-\epsilon_2}g^{-1}$ is either 
$s_{\epsilon_2-\epsilon_3}$ or $s_{\epsilon_2+\epsilon_3}$. If the 
latter case occurs we replace $G$ by 
$s_{\epsilon_{3}}Gs_{\epsilon_{3}}$ and obtain 
$s_{\epsilon_2-\epsilon_3}\in G$. Repeating this argument with $(34), 
(45)$ etc. we obtain that replacing $G$ by some $wGw^{-1}$ we may 
assume that $G\supset \langle 
s_{\epsilon_1-\epsilon_2},\ldots,s_{\epsilon_{n-1}-\epsilon_n}\rangle 
= G_{1}$. Suppose $G$ contains the pair $s_{\epsilon_i-\epsilon_j}, 
s_{\epsilon_i+\epsilon_j}$ for some $i\neq j$. Let $1\leq k<\ell\leq 
n$. Let $g\in G$ be an element such that the permutation 
$\overline{g}\in S_n$ transforms $i\mapsto k, j\mapsto \ell$. Then 
$g\{s_{\epsilon_i-\epsilon_j},s_{\epsilon_i+\epsilon_j}\}s^{-1} = 
\{s_{\epsilon_k-\epsilon_{\ell}},s_{\epsilon_k+\epsilon_{\ell}}\}$. 
Therefore $G$ contains $s_{\beta}$ for $\forall \beta \in R_{\ell}$. 
Hence $G=W(B_n)$ or $G=W(D_n)$. If no pair 
$\{s_{\epsilon_i-\epsilon_j}, s_{\epsilon_i+\epsilon_j}\}$ is 
contained in $G$, then the only reflections contained in $G$ are 
$\{s_{\epsilon_i-\epsilon_j}|1\leq i<j\leq n\}$. Hence if $g\in G$, 
then for $\forall i\neq j$ one has 
$gs_{\epsilon_i-\epsilon_j}g^{-1}=s_{\epsilon_k-\epsilon_{\ell}}$ for 
some $k\neq \ell$. Therefore $G\subset N_{W(B_n)}(G_1)$. Since 
$|N_{W(B_n)}(G_1):G_{1}|=2$ one has that either $G=G_{1}$ or 
$G=N_{W(B_n)}(G_1)$. If $n=2$, then $N_{W(B_2)}(G_1)=W(D_2)$. If 
$n\geq 3$ we apply Lemma~\ref{s0.52a}.
\end{proof}
\begin{pro}\label{s0.54}
Let $p:C\overset{\pi}{\lto}C'\overset{g}{\lto}Y$ be a $B_n$-covering 
with irreducible $C$. Suppose $g:C'\to Y$ is primitive (cf. 
\S\ref{s0.50}). Furthermore suppose that $g:C'\to Y$ is ramified, in one of the branch 
points $b\in \mathfrak{D}\subset Y$ it is simply 
ramified and $\pi$ is unramified in $g^{-1}(b)$. Then either 
$G=W(B_n)$, or $G=W(D_n)$, or $G$ is conjugated to $N_{W(B_n)}(G_1)$ 
in case $n\geq 3$. The latter case happens if and only if $p:C\to Y$ 
fits into a commutative diagram
\begin{equation} \label{es0.54} \xymatrix{ 
& C \ar[dl]_{\pi}\ar[dr]^{f} &\\ 
C' \ar[dr]_{g} & & \tilde{Y} \ar[dl]^{h} \\ 
& Y  }
\end{equation}
where $h:\tilde{Y}\to Y$ is a covering of degree 2.
\end{pro}
\begin{proof}
The possible alternatives for $G$ follow from Lemma~\ref{s0.53}. In 
case (iii), renumbering the fiber $p^{-1}(y_{0})$, one may assume 
$G=N_{W(B_n)}(G_1)$. Then one has a commutative diagram of 
$G$-equivariant maps
\begin{equation}\label{es0.55} \xymatrix{
 &  \{\pm \epsilon_j|1\leq j\leq n\} \ar[dl]\ar[dr]  & \\  
\{a_j|1\leq j\leq n\} \ar[dr] & &  \{\Sigma_1,\Sigma_2\} \ar[dl]  \\   
&  \star    }
\end{equation}
This yields \eqref{es0.54}. Viceversa, suppose \eqref{es0.54} holds 
for some $f:C\to \tilde{Y}$ and $h:\tilde{Y}\to Y$. Let 
$\tau:\tilde{Y}\to \tilde{Y}$ be the involution such that $h\circ \tau 
= h$. For every $x\in C$ one has $h\circ f(x)=h\circ f(i\, x)$ since 
$h\circ f=g\circ \pi$. It is impossible that $f(i\, x)=f(x)$ for 
$\forall x \in C$, since this would imply a decomposition $C'\to 
\tilde{Y}\to Y$, while by hypothesis $C'\to Y$ is primitive. Therefore 
$f(i\, x)=\tau f(x)$. Now, one can number the points of 
$p^{-1}(y_{0})$ so that a diagram \eqref{es0.55}, with maps commuting 
with the monodromy action, takes place. Hence $G$ is conjugated to 
$N_{W(B_n)}(G_1)$. 
\end{proof}
When $g:C'\to Y$ is simply ramified, primitivity is equivalent to the 
surjectivity of $g_{*}:\pi_{1}(C',*)\to \pi_{1}(Y,*)$. So, we obtain 
the following corollary.
\begin{cor}\label{s0.56}
Let $p:C\overset{\pi}{\lto}C'\overset{g}{\lto}Y$ be a simply ramified 
$B_n$-covering with irreducible $C$. Assume $g_{*}:\pi_{1}(C',*)\to 
\pi_{1}(Y,*)$ is 
surjective. Let $G$ with $G\subset W(B_n)\subset S_{2n}$ be the 
monodromy group of the covering $p:C\to Y$. 

If $\mathfrak{D}_{s}\neq \emptyset$, then $G=W(B_n)$. 

If $\mathfrak{D}_{s}=\emptyset$, then 
one of the following alternatives holds: either $G=W(B_n)$; or $G=W(D_n)$; or 
$G$ is conjugate to $N_{W(B_n)}(G_1)$ in case $n\geq 3$. The latter 
alternative holds if and only if 
$p:C\overset{\pi}{\lto}C'\overset{g}{\lto}Y$ fits into a commutative 
diagram \eqref{es0.54}  with \'{e}tale $\tilde{Y}\to Y$. 
\end{cor}
\begin{proof}
$\mathfrak{D}_{\ell}\neq \emptyset$ since otherwise 
$g_{*}:\pi_{1}(C',*)\to \pi_{1}(Y,*)$ would not be surjective. The 
subgroups $W(D_n)$ and $N_{W(B_n)}(G_1)$ do not contain reflections 
with respect to short roots $s_{\epsilon_{j}}$, so only 
$G=W(B_n)$ is possible if $\mathfrak{D}_{s}\neq \emptyset$. Let 
$\mathfrak{D}_{s}=\emptyset$. By Lemma~\ref{s0.52a} the reflections 
which belong to $N_{W(B_n)}(G_1)$ do not interchange $\Sigma_{1}$ and 
$\Sigma_{2}$. Therefore, if $G$ is conjugate to $N_{W(B_n)}(G_1)$, the 
covering $\tilde{Y}\to Y$ is \'{e}tale.
\end{proof}
\begin{cor}\label{s0.57}
Let $p:C\overset{\pi}{\lto}C'\overset{g}{\lto}\mathbb{P}^1$ be a 
simply ramified $B_n$-covering with irreducible $C$. Let $G\subset 
W(B_n)\subset S_{2n}$ be 
the monodromy group of $p:C\to \mathbb{P}^1$. 

If $\mathfrak{D}_{s}\neq \emptyset$, then $G=W(B_n)$. 
If $\mathfrak{D}_{s}=\emptyset$, then $G=W(D_n)$. 
\end{cor}
\begin{rem}\label{ptv1.1}
We notice that when $g(Y)\geq 1$ the monodromy group of $p:C\overset{\pi}{\lto}C'\overset{g}{\lto}Y$
might very well be $W(B_n)$, even when $\pi :C\to C'$ is \'{e}tale. The fundamental group 
$\pi_1(Y\setminus \mathfrak{D},y_0)$ is generated by $\gamma_1,\ldots,\gamma_n,\alpha_1,\beta_1,\ldots,\alpha_g,\beta_g$,
where $\gamma_1,\ldots,\gamma_n$ are homotopy classes of loops encircling the branch points of $p:C\to Y$, with the only relation
\begin{equation}\label{eptv1.1}
\gamma_1\cdots\gamma_n = [\alpha_1,\beta_1]\cdots [\alpha_g,\beta_g].
\end{equation}
When $\pi:C \to C'$ is \'{e}tale the monodromy map $m:\pi_1(Y\setminus \mathfrak{D},y_0)\to W(B_n)$ has the property that $m(\gamma_i)\in W(D_n)$ for $\forall i=1,\ldots,n$ (cf. \cite{K}, pp.179,180), but the only restriction for 
$m(\alpha_j),m(\beta_j)\in W(B_n)$ comes from the relation \eqref{eptv1.1}. Reversing and applying Riemann's existence theorem it is easy to construct, when $g(Y)\geq 1$, coverings of type $B_n$ with \'{e}tale $\pi:C \to C'$ and full monodromy group  $W(B_n)$.
\end{rem}
\begin{dfn}\label{s0.57a}
Let $p:C\overset{\pi}{\lto}C'\overset{g}{\lto}Y$  be a covering with $\deg(\pi)= 2$ and 
irreducible $C$.
We say it is a {\it simple 
$B_n$-covering} if:
\renewcommand{\theenumi}{\roman{enumi}}
\begin{enumerate}
\item 
$p$ is simply ramified $B_n$-covering and both $\pi$ and $g$ are ramified;
\item 
$g:C'\to Y$ is primitive, or equivalently $g_{*}:\pi_{1}(C',*)\to 
\pi_{1}(Y,*)$ is surjective;
\end{enumerate}
\end{dfn}
\begin{rem}\label{ptv1.2}
Notice that if $\deg(g)$ is prime, then $g:C'\to Y$ is primitive. So in 
this case the simply ramified $B_n$-coverings with $\mathfrak{D}_s\neq \emptyset$ and
$\mathfrak{D}_{\ell}\neq \emptyset$
 are simple $B_n$-coverings. The same statement holds when $\deg(g)$ is arbitrary and
  $Y\cong \mathbb{P}^1$.
 By Corollary~\ref{s0.56} every simple 
$B_n$-covering has full monodromy group $W(B_n)$.
\end{rem}
\begin{block}\label{s0.3}
The short root $\epsilon_{1}$ is the fundamental weight $\omega_{1}$ 
of the root system $R$ of type $B_{n}$ (cf. \cite{B}). Let us consider 
the fundamental weight $\omega_{n}=\frac{1}{2}(\epsilon_{1}+\cdots 
+\epsilon_{n})$. This is the dominant weight of 
the spinor representation of the Lie algebra
 $so(2n+1)$. We call $\omega_n$ the spinor weight. One has 
\begin{equation}\label{es0.3a}
W\omega_{n} = \{\lambda_{A}=\frac{1}{2}(\sum_{j\notin 
A}\epsilon_{j}-\sum_{j\in A}\epsilon_{j})\;|\;A\subset\{1,\ldots,n\}\}
\end{equation}
Let $m:\pi_1(Y\setminus \mathfrak{D},y_0) \to W\subset S_{2n}$ be the 
monodromy map of the covering 
$p:C\overset{\pi}{\lto}C'\overset{g}{\lto}Y$ as in \S\ref{s0.1}.
Composing $m$ with the permutation representation $W\to 
S(W\omega_{n})$ and applying Riemann's existence theorem one obtains a 
smooth, projective curve $X$ and a
covering $f:X\to Y$ of degree $2^{n}$. Let us fix a bijection 
$f^{-1}(y_0)\overset{\sim}{\lto}W\omega_{n}$. If $y\in Y\setminus 
\mathfrak{D}$ and $\gamma$ is a path in $Y\setminus \mathfrak{D}$ 
which connects $y$ with $y_0$, one enumerates the points of $f^{-1}(y)$ 
using covering homotopy along $\gamma$ and the fixed bijection. 
Thus $f^{-1}(y)=\{x_{A}|A\subset \{1,\ldots,n\}\}$ in correspondence with \eqref{es0.3a}.
 
Calculating the action of 
$s_{\epsilon_{j}}$ and $s_{\epsilon_{j}\pm \epsilon_{k}}$ on 
$W\omega_{n}$ one obtains that if 
$p:C\overset{\pi}{\lto}C'\overset{g}{\lto}Y$ 
is a simply ramified $B_n$-covering, then the local monodromy of $f:X\to Y$ is 
a product 
of $2^{n-1}$ independent transpositions at a point 
$b\in \mathfrak{D}_{s}$ and a product of 
$2^{n-2}$ independent transpositions at a point $b\in \mathfrak{D}_{\ell}$.

One calculates a correspondence $D$ as in \S\ref{s0.14} letting 
$L=Q(R)=\oplus_{j=1}^{n}\mathbb{Z}\epsilon_{j}$, 
$(\epsilon_{j}|\epsilon_{k})=-2\delta_{jk}$ and $\lambda=\omega_{n}$. 
It is shown in \cite{K}, Section~8.8 that one has for $D$ the formula
\begin{equation} \label{eqn15}
D(x_{A}) = \sum_{B\neq A}(|A|+|B|-2|A\cap B|-1)x_{B}
\end{equation}
and furthermore $q=2^{n-1}$ and $\deg D = 2^{n-1}(n-2)+1$. Let 
$\delta:JX\to JX$ be the endomorphism induced by $D$, let 
$$
A=Ker(Nm_{f}:JX\to JY)^{0} \quad \mbox{and} \quad P(X,\delta)=(1-\delta)A.
$$  
Applying Proposition~\ref{s0.16} to $W=W(B_n),\lambda = \omega_{n}$ 
and $\lambda'=\omega_{1}$, one obtains that $P(X,\delta)$ and 
$P(C,i)=P(C,C')$ are isogenous. In the next paragraph we want to 
describe this isogeny more explicitely.
\end{block}
\begin{block}\label{s0.21}
The correspondence $S\in Div(X\times C)$ which establishes the isogeny 
between $P(X,\delta)$ and $P(C,C')$ is constructed in \cite{K}, 
Section~8.8. We recall this briefly. One considers the bilinear form 
$(\epsilon_{j}|\epsilon_{k})=-2\delta_{jk}$. Here 
\[
\mathbb{R}^{d}=\oplus_{A\subset\{1,\ldots,n\}}\mathbb{R}e_{A},\qquad 
\mathbb{R}^{e} = \oplus_{j=1}^{n}(\mathbb{R}f_{j}\oplus 
\mathbb{R}f_{-j})
\]
One has 
\[
S(e_{A}) = 2S_{0}(e_{A}) + nT(e_{A})
\]
where
\[
S_{0}(e_{A}) = \sum_{j\notin A}f_{-j} + \sum_{j\in A}f_{j}, \qquad 
T(e_{A}) = \sum_{j=1}^{n}(f_{j}+f_{-j})
\]
Since $S$ and $T$ are $W$-equivariant, so is $S_{0}$. The construction 
of \S\ref{s0.11} yields correspondences in $Div(X\times C)$, which 
abusing notation we denote again by $S,S_{0}$ and $T$. One has for 
every $y\in Y\setminus \mathfrak{D}$ that 
$f^{-1}(y)=\{z_{A}|A\subset\{1,\ldots,n\}\}$, 
$p^{-1}(y)=\{x_{j},x'_{j}|j=1,\ldots,n\}$, with 
$x'_{j}=i(x_{j}):=x_{-j}$, and $S(z_{A})=2S_{0}(z_{A})+nT(z_{A})$, 
where
\[
S_{0}(z_{A}) = \sum_{j\notin A}x'_{j} + \sum_{j\in A}x_{j}, \qquad 
T(z_{A})=p^{*}(f(z_{A})).
\]
Replacing $S$ by $-S$ one has 
$-S(z_{A})=2S_{1}(z_{A})-(n+1)T(z_{A})$, where
\[
S_{1}(z_{A}) = \sum_{j\notin A}x_{j} + \sum_{j\in A}x'_{j}
\]
We may use the correspondence $S_{1}$ in order to fit the covering $f:X\to Y$
 into the 
following commutative diagram:
\begin{equation}\label{es0.3b}
\begin{CD}
X         @>S_{1}>>         C^{(n)}\\
@VfVV                   @VV\pi^{(n)}V\\
Y        @>g^*>>       C'^{(n)}
\end{CD}
\end{equation}
Here $g^*$ denotes the map associating to a point $y \in Y$ the whole fibre 
$g^{-1}(y)$ 
considered as a point of the $n$-fold symmetric product $C'^{(n)}$ of $C'$.
It is clear that $S_{1}: X\to C^{(n)}$ is generically injective.
We will later need some properties of $S_{0}$ which we now prove.
\end{block}
\renewcommand{\theenumi}{\roman{enumi}}
\begin{lem}\label{s0.23}
Let $s_{0}:JX\to JC$ and ${^{t}s}_{0}:JC\to JX$ be the homomorphisms 
induced by $S_{0}$ and ${^{t}S}_{0}$ respectively. Let $T_{1}\in 
Div(X\times X)$, $T\in Div(X\times C)$ and $T_{2}\in Div(C\times C)$ be 
the trace correspondences (cf. \S\ref{s0.11}). Then 
\begin{enumerate}
\item 
$s_{0}$ and ${^{t}s}_{0}$ induce homomorphisms 
$s_{0}:Ker(Nm_{f})^{0}\to P(C,C')$, 
${^{t}s}_{0}:Ker(Nm_{p})^{0}\to 
P(X,\delta)$ and isogenies $s_{0}:P(X,\delta)\to P(C,C')$, 
${^{t}s}_{0}:P(C,C')\to P(X,\delta)$, such that 
${^{t}s}_{0}s_{0}|_{_{P(X,\delta)}}=2^{n-1}id_{_{P(X,\delta)}}$, 
$s_{0}{^{t}s}_{0}|_{_{P(C,C')}}=2^{n-1}id_{_{P(C,C')}}$
\item 
${^{t}S}_{0}(S_{0}(z_{A}))=z_{A}-D(z_{A})+(n-1)T_{1}(z_{A})$ for 
$\forall A\subset \{1,\ldots,n\}$
\item
$S_{0}({^{t}S}_{0}(x_{k}))=2^{n-2}(x_{k}-i(x_{k}))+2^{n-2}T_{2}(x_{k})
$ for $\forall k \in [-n,n]$
\end{enumerate}
\end{lem}
\begin{proof}
Part~(i) follows from Proposition~\ref{s0.16} and its proof, observing 
that $s=2s_{0},\; {^{t}s}=2\, {^{t}s}_{0}$. All the correspondences 
are obtained from linear maps as in \S\ref{s0.11}, so (abusing 
notation) we need to verify the following equalities of linear maps:
\[
{^{t}S}_{0}\cdot S_{0} = E-D +(n-1)T_{1}, \qquad S_{0}\cdot 
{^{t}S}_{0} = 2^{n-2}(E-I)+2^{n-2}T_{2}
\]
where $I$ is the involution $I(f_{k})=f_{-k}$ and $E$ is the identity map. 
We use \cite{K}, 
Lemma~6.5.1 (we notice that the identity map is denoted by $I$ in that paper). The data $W(B_{n}), 
L=\oplus_{j=1}^{n}\mathbb{Z}\epsilon_{j}, \lambda=\omega_{n}$ and 
$(\epsilon_{j}|\epsilon_{k})=-2\delta_{jk}$ yields $G=D-E$ and 
$q=2^{n-1}$. 
The data $W(B_{n}), 
L=\oplus_{j=1}^{n}\mathbb{Z}\epsilon_{j}, 
\lambda'=\omega_{1}$ and 
$(\epsilon_{j}|\epsilon_{k})=-2\delta_{jk}$ yields $G'=2(I-E)+T_{2}$ and 
$q'=4$(cf. \cite{K},\S8.8.1 and \S4.7). According to \cite{K}, 
Lemma~6.5.1 one has 
\[
{^{t}S}\cdot S = -q'G + d_{1}T_{1} = 4(E-D) + d_{1}T_{1}
\]
Replacing $S$ by $2S_{0}+nT$ and using Lemma~\ref{s0.10a} one obtains 
\[
{^{t}S}_{0}\cdot S_{0} = E-D + f_{1}T_{1}
\]

The degrees of the correspondences $S_{0},{^{t}S}_{0},D-E$ and $T_{1}$ 
are $n,2^{n-1}$, \newline
$2^{n-1}(n-2)$ and $2^{n}$ respectively. Therefore 
$f_{1}=n-1$ and Equality~(ii) is verified. Equality~(iii) is proved 
similarly using $S\cdot {^{t}S}=-qG'+d_{2}T_{2}$.
\end{proof}

\section{Some results for arbitrary $n$}\label{s2}

\begin{block}\label{s2.1}
Let the situation be as in \S\ref{s0.1}. So $p: C \to Y$ with 
decomposition (\ref{e0.2}) 
denotes a covering of type $B_n$. Assume moreover that $p$ is a simple $B_n$-covering.
In particular $C$ is an irreducible curve. 
According to \S\ref{s0.1} Hurwitz formula gives
\[
g(C') = \frac{|\mathfrak{D}_l|}{2} + ng(Y) -n+1, \qquad  g(C)= 
\frac{|\mathfrak{D}_s|}{2} + 
|\mathfrak{D}_l| + 2ng(Y) -2n +1.
\] 
Hence the Prym variety $P(C,C')$ of the covering $\pi:C \to C'$ is of 
dimension
\begin{equation}\label{es2.1a}
\dim P(C,C') = g(C') + \frac{|\mathfrak{D}_s|}{2} -1 = 
\frac{|\mathfrak{D}_s|+|\mathfrak{D}_l|}{2} +n 
g(Y) -n.
\end{equation} 
Consider the covering $f:X \to Y$ defined in \S\ref{s0.3}. 
  The curve $X$ is smooth by construction. It is connected, and therefore irreducible, because by Corollary~\ref{s0.56}
the monodromy group of $p:C\to Y$ is $W(B_n)\subset S_{2n}$ and therefore the monodromy group of
$f:X\to Y$ is a transitive subgroup of $S_{2^n}$.
 Since the covering $f$ is of degree $2^n$, we obtain
\begin{equation}\label{es2.1b}
g(X) = 2^{n-2}|\mathfrak{D}_s|+2^{n-3}|\mathfrak{D}_l| + 2^ng(Y) - 2^n +1.
\end{equation}
Counting the number of points over  branch points in $Y$ in the diagram  \eqref{es0.3b}
it is easily seen that $S_1:X\to C^{(n)}$  is injective, taking into account that  $g^*$ is injective.  Hence we may consider the points of $X$
as points of the symmetric product $C^{(n)}$.  In other words, we denote the points of $X$ by 
$x = x_1 + \cdots + x_n$ with $x_i \in C$.

The curve $X$ admits an involution $\sigma$, induced by the involution $'$ 
defined by the double covering $\pi$, 
namely 
$$
\sigma(x_1 + \cdots + x_n) =  x'_1 + \cdots + x'_n.
$$
Note that for $n \geq 3$ the involution $\sigma$ is fixed-point free. Hence, 
denoting $X' = X/\sigma$ and by $P(X,X')$ 
the corresponding Prym variety, we have
\begin{equation}\label{es2.1c}
\dim P(X,X') = \frac{g(X)-1}{2} = 
2^{n-3}|\mathfrak{D}_s|+2^{n-4}|\mathfrak{D}_l|  + 
2^{n-1}(g(Y) -1).
\end{equation}
\end{block}

\begin{block} \label{s2.2}
For any $x = x_1 + \cdots + x_n \in X$ and any subset $A \subset \{1, \ldots, 
n\}$ we denoted 
$x_A = \sum_{i\notin A} x_i + \sum_{i \in A} x'_i.$ In particular, $x = 
x_{\emptyset}$.
In \S\ref{s0.3} we defined a correspondence $D$ on $X$ and we saw that it is given by 
equation \eqref{eqn15}.
In particular we have for $x = x_{\emptyset}$,
$$
\begin{array}{cl}
D(x) & = \sum_{|B|\geq 2}(|B|-1)x_B\\
& = \sum_{|B|=2}x_B + 2 \sum_{|B|=3}x_B + \cdots + (n-1)x_{\{1,\ldots,n\}}
\end{array}
$$
This implies $\deg D = 2^{n-1}(n-2)+1.$
Moreover $D$ is of exponent $q = 2^{n-1}$ (see \S\ref{s0.3}).
\end{block}

\begin{lem} \label{lem2.3}
The correspondence $D$ commutes with $\sigma$:
$\sigma D = D \sigma.$
\end{lem}

\begin{proof}
Note first that for all subsets $B \subset \{1,\ldots,n\}$ we have 
$\sigma(x_B)=x_{\overline B}$, where 
${\overline B} = \{1, \ldots,n\} \setminus B$. Hence 
$$
D\sigma(x_A) = D(x_{{\overline A}}) = \sum_{C \neq {\overline A}}(|{\overline 
A}| + |C| - 2|{\overline A} \cap C| -1)x_C
$$
Setting $C = {\overline B}$ and noting that $|{\overline A}|+|{\overline B}|-2|{\overline A} \cap {\overline B}| = 
|A|+|B|-2|A \cap B|,$
we have
$$
\begin{array}{cl}
D\sigma(x_A) &= \sum_{B \neq A}(|{\overline A}|+|{\overline B}|-2|{\overline 
A} \cap {\overline B}|-1)x_{\overline B}\\
& = \sum_{B \neq A}(|A| + |B| - 2|A \cap B| -1)x_{\overline B}\\
& = \sigma D(x_A). 
\end{array}
$$
\end{proof}
\begin{lem} \label{lem2.4}
For any $x = x_{\emptyset} \in X$,
$$
(D-1)(x + \sigma(x)) = (n-2) f^*f(x).
$$
\end{lem}
\begin{proof}
Writing $\Sigma = \{1,\ldots,n\}$ we have using Lemma \ref{lem2.3},
$$
\begin{array}{cl}
(D-1)(x + \sigma(x))&=\sigma(Dx_{\emptyset} - x_{\emptyset}) + 
(Dx_{\emptyset}-x_{\emptyset})  \\
&= \sigma(\sum_{|B| \geq 2}(|B|-1)x_B - x_{\emptyset}) + \sum_{|B| \geq 
2}(|B|-1)x_{B} - x_{\emptyset}\\
&= \sum_{|B| \geq 2}(|B|-1)x_{\overline B} - x_{\Sigma} + \sum_{|B| \geq 
2}(|B|-1)x_{B} - x_{\emptyset}\\
&= \sum_{B \subset \Sigma}(|B|-1)x_{\overline B} + \sum_{B \subset 
\Sigma}(|B|-1)x_{B}\\
&= \sum_{B \subset \Sigma}(n-|B|-1)x_{B} + \sum_{B \subset 
\Sigma}(|B|-1)x_{B}\\
&= (n-2) \sum_{B \subset \Sigma}x_B = (n-2)f^*f(x)
\end{array}
$$
\end{proof}

As in \S\ref{s0.3} consider $A = Ker(Nm_f:JX \ra JY)^0$, 
the connected component of the kernel of the norm map of $f$. If 
$\delta \in End(JX)$ denotes the endomorphism 
induced by the correspondence $D$, we defined the Prym-Tyurin variety 
associated to $D$ by
\[
P(X,\delta) = (1- \delta)A.
\]
Lemma \ref{lem2.3} implies that $D$ induces an endomorphism of the 
Prym variety $P(X,X')$, also denoted bt $\delta$. Using this notation we have

\begin{pro} \label{prop2.5}
$P(X,\delta)$ and  $P(X,X')$ are related as follows:
\renewcommand{\theenumi}{\roman{enumi}}
\begin{enumerate}
\item 
$P(X,\delta) \subset P(X,X')$,
\item 
$P(X,\delta) = (\delta - 1)P(X,X')$.
\end{enumerate}
\end{pro}
\noindent

\begin{proof}
(i) follows immediately from 
Lemma \ref{lem2.4}. 
For the proof of (ii)
note first that $P(X,X') \subset A$, since $Nm_f(x - \sigma(x)) = 0$ for every 
$x \in X$. Moreover,
$JX = P(X,X') + \pi^*JX'$. Intersecting with $A$ gives
\[
A = P(X,X') + (\pi^*JX' \cap A)^0.
\]
On the other hand, $P(X,\delta) = (\delta - 1)A$ by definition. Hence it suffices 
to check that 
$(\delta -1)((\pi^*JX' \cap A)^0) =0$. 

But $\pi^*JX' \cap A = \{\mathfrak{a}' + \sigma(\mathfrak{a}') \;|\; 
\mathfrak{a}' \in JX', \,  
Nm_f(\mathfrak{a}' + \sigma(\mathfrak{a}')) = 0\}$
and we have $Nm_f(\mathfrak{a}' + \sigma(\mathfrak{a}')) = 2 \, 
Nm_f(\mathfrak{a}')$. This implies, using Lemma \ref{lem2.3} 
and the proof of Lemma \ref{lem2.4}, 
\[
2(D-1)(\mathfrak{a}' + \sigma(\mathfrak{a}')) = 2(n-2)f^* Nm_f(\mathfrak{a}') 
= 0.
\]
Hence $(\delta -1)(\pi^*JX' \cap A)$ consists of torsion points, which implies 
$(\delta -1)((\pi^*JX' \cap A)^0) = 0.$ 
\end{proof}
\begin{block}\label{s2.6}
We need a construction due to Donagi (cf. \cite{D2}). 
Consider again the spinor weight $\omega_n$. The subgroup 
$W(D_n)\subset W(B_n)=W$ is of index 2 and one has a splitting 
(cf. \eqref{es0.3a})
\begin{equation}\label{e33.15}
W\omega_{n} = \{\lambda_{A}|\: |A|\; \text{is even}\}\cup 
\{\lambda_{A}|\: |A|\; \text{is odd}\}.
\end{equation}
The subgroup $W(D_n)$ acts transitively on the two subsets on the 
right. We notice that the latter subsets are the orbits of two
fundamental weights of the root system of type $D_n$, namely the 
dominant weights associated to the two semispinor representations 
of the Lie algebra $so(2n)$. 
In the situation of \S\ref{s0.3} the monodromy map of the 
covering $f:X\to Y$ is $\pi_1(Y\setminus \mathfrak{D},y_0) 
\overset{m}{\lto}W(B_n)\eto S(W\omega_{n})$. Let 
$$
G=m(\pi_1(Y\setminus \mathfrak{D},y_0))
$$ 
be the monodromy group. The splitting 
\eqref{e33.15} yields a decomposition 
$$
f:X\overset{g}{\lto}\tilde{Y}\overset{h}{\lto}Y.
$$ 
Here $h$ is of 
degree 2 and its monodromy map is the composition of $m$ with 
$W(B_n)\to S_{2}$. The latter is obtained from $W(B_n)$ acting on the 
two subsets of \eqref{e33.15}. 

If $G\subset W(D_n)$, then $\tilde{Y}$ 
is a disjoint union $\tilde{Y}=Y_{1}\sqcup Y_{2}$ of two copies of $Y$ 
and respectively $X=X_{1}\sqcup X_{2}$. If $G\not \subset W(D_n)$, 
then $\tilde{Y}$ is irreducible and the monodromy group of $g:X\to 
\tilde{Y}$ equals $G\cap W(D_n)$. 

If $p:C\to C'\to Y$ is a simply 
ramified $B_n$-covering then, calculating the action of the 
reflections on $W\omega_{n}$, we see that $h:\tilde{Y}\to Y$ is 
ramified in $\mathfrak{D}_{s}$ and $g:X\to \tilde{Y}$ is ramified in 
$h^{-1}(\mathfrak{D}_{\ell})$. 
The  Prym variety 
$P(\tilde{Y},Y)$ is of dimension
\begin{equation}\label{es2.6a}
\dim P(\tilde{Y},Y) = \frac{|\mathfrak{D}_s|}{2} +g(Y) -1.
\end{equation}
Using the correspondence $S_{1}\in Div(X\times C)$ of \S\ref{s0.21} and the commutative 
diagram \eqref{es0.3b} we may give another interpretation of the map $g:X\to 
\tilde{Y}$. 
Two points  
$x = x_1 + \cdots + x_n$ and $\tilde{x} = \tilde{x}_1 + \cdots +\tilde{x}_n$ 
of $X$ are called equivalent, denoted $x \sim \tilde{x}$, if and only if 
$f(x) = f(\tilde{x})$ and the points $x_i$ and $\tilde{x}_i$ of $C$ differ by 
an even number of changes. Then $\tilde{Y}$ is the quotient of $X$ 
modulo this equivalence relation and $g:X\to \tilde{Y}$ is 
the map associating to any $x = x_{\emptyset} \in X$ 
the equivalence class $\{ x_A\;|\; |A|\, even \}$.

 For $n$ even, the 
involution 
$\sigma$ respects the equivalence relation and $f$ factorizes as follows
\[
X \to X' \to \tilde{Y} \to Y.
\]
For $n$ odd, the involution $\sigma$ exchanges the equivalence classes and  we 
have instead a commutative diagram
\[
\begin{CD}
X         @>g>>         \tilde{Y}\\
@VVV                   @VVhV\\
X'        @>>>       Y
\end{CD}
\]
Let $\tau: \tilde{Y} \to \tilde{Y}$ denote the involution associated to the 
covering $h$. Then $g(\sigma(x)) = \tau(g(x))$ 
implies that $g$ induces a homomorphism of Prym varieties $Nm_g: P(X,X') \to 
P(\tilde{Y},Y)$.
\end{block}
\begin{lem} \label{lem2.6}
$Nm_g: P(X,X') \to P(\tilde{Y},Y)$ is the zero map for $n$ even and surjective 
for $n$ odd.
\end{lem}
\begin{proof}
For even $n$ we have $g(x_{\emptyset}) = g(\sigma(x_{\emptyset}))$ for any $x 
= x_{\emptyset} \in X$, 
which implies the assertion. The surjectivity for odd $n$ is obvious.
\end{proof}
\begin{pro} \label{prop2.7}
For any odd integer $n \geq 3$,
\renewcommand{\theenumi}{\roman{enumi}}
\begin{enumerate}
\item
\(
P(X,\delta) \subset K := Ker(P(X,X') 
\stackrel{Nm_g}{\longrightarrow} 
P(\tilde{Y},Y))^{0}.
\)
\item 
$g^{*}P(\tilde{Y},Y)\subset (\delta +2^{n-1}-1)P(X,X')$
\end{enumerate}
\end{pro}
\begin{proof}

(i)\;  
For any $x = x_{\emptyset} \in X$ we have $D(x_{\emptyset}) = \sum_{k=2}^n 
\sum_{|A|=k} (k-1)x_A$, which implies 
\[
g(D(x_{\emptyset}) -x_{\emptyset}) = [\sum_{\stackrel{k=0}{k \; even}}^n 
\sum_{|A|=k} (k-1)] g(x_{\emptyset}) + 
[\sum_{\stackrel{k=0}{k \; odd}}^n \sum_{|A|=k} (k-1)] \tau(g(x_{\emptyset}))
\]
Define 
\[
M := \sum_{\stackrel{k=0}{k \; even}}^n \sum_{|A|=k} (k-1) = 
\sum_{\stackrel{k=0}{k \; odd}}^n \sum_{|A|=k} (k-1)
\]
and note that the equality is an immediate consequence of the well known 
binomial identities $\sum_{k=0}^n (-1)^k{n \choose k} = 0$ and 
$\sum_{k=0}^n (-1)^{k}k{n \choose k} = 0$.
We get 
\[
g(D(x_{\emptyset}) -x_{\emptyset}) = M(g(x_{\emptyset}) + \tau 
g(x_{\emptyset}))
\]
and thus for all $\mathfrak{a} \in JX$
\[
g(\delta(\mathfrak{a}) - \mathfrak{a}) = M [ Nm_g(\mathfrak{a}) + \tau 
Nm_g(\mathfrak{a})] = M h^* Nm_f(\mathfrak{a}).
\]
Now recall that $P(X,\delta) = (\delta - 1) A$. So any $\mathfrak{b} \in 
P(X,\delta)$ is of the form
$\mathfrak{b} = (\delta - 1)(\mathfrak{a})$ with $Nm_f(\mathfrak{a}) = 0$. 
This implies
$Nm_g(\mathfrak{b}) = M h^* Nm_f(\mathfrak{a}) = 0$ and thus the 
assertion.

(ii)\; If $z\in \tilde{Y}$, then 
$g^{*}(z-\tau(z))=g^{*}(z)-\sigma(g^{*}(z))$. 
This shows that $g^{*}(P(\tilde{Y},Y))\subset P(X,X')$. 

Let $JX=V/\Lambda, 
J\tilde{Y}=W/\Gamma$, and let $V^{-}, W^{-}$ be the anti-invariant 
subspaces of $\sigma$ and $\tau$ respectively. Since 
$E_{JX}=-(\;,\,)_{X}$ and $E_{J\tilde{Y}}=-(\;,\,)_{\tilde{Y}}$, we have 
by the projection formula,
$E_{J\tilde{Y}}(g_{*}v,w)=E_{JX}(v,g^{*}w)$ for any $v\in V$ and $w\in 
W$. So, $g^{*}(W)$ is orthogonal to $Ker(g_{*}:V\to W)$. We saw that 
$g^{*}(W^{-})\subset V^{-}$. By Proposition~\ref{prop2.5} (i) and Part~(i) 
one has $(1-\delta)V^{-}\subset 
V^{-}\cap Ker(g_{*})$. The subspace $(\delta +q-1)V^{-}$, with $q=2^{n-1}$, is 
the orthogonal complement of $(1-\delta)V^{-}$ with respect to 
$E_{\Xi}= \frac{1}{2}(E_{JX}|_{V^{-}})$. Therefore 
$g^{*}(W^{-})\subset (\delta+2^{n-1}-1)V^{-}$, which proves (ii).
\end{proof}
\begin{rem}\label{ptv1.5}
We notice that the assumption  $p:C\to Y$ is a simple $B_n$-covering, which we made for simplicity at the 
beginning of this section, was essentially used only in \S\ref{s2.1} for calculating various dimensions using the 
Hurwitz formula. The constructions and the proofs of the statements in \S{\S}\ref{s2.2}--\ref{prop2.7} 
hold under the more general assumption
that $p:C\to Y$ is an arbitrary $B_n$-covering, in particular the nonsingular curves $C$ and $X$ might be reducible.
Indeed the proofs of the statements are based on identities between various correspondences. One verifies these identities 
on the generic fibers of $p:C\to Y$ and $f:X\to Y$, as done in the text above, then taking closures
obtains the identities on the whole curves. 
\end{rem}

\section{Dual Abelian varieties and the spinor weight}\label{s0a}
\begin{block}\label{s0.5}
We first recall some material from \cite{BL3} (see also \cite{BL2}, 
Section~14.4 and \cite{K1}, \S3.4). Let $(P,L)$ be  a polarized 
Abelian variety, $P=V/\Lambda$, let $E:\Lambda \times \Lambda \to 
\mathbb{Z}$ be the Riemann form and let $(d_1,\ldots,d_{p})$ be the 
polarization type of $L$. One has $V\cong 
\Lambda_{\mathbb{R}}:=\Lambda\otimes _{\mathbb{Z}}\mathbb{R}$ as 
$\mathbb{R}$-vector spaces. Let $I:\Lambda_{\mathbb{R}}\to 
\Lambda_{\mathbb{R}},\; I^{2}=-id$ be the operator defining the complex 
structure of $V$. The dual Abelian variety $\hat{P}$ is isomorphic to 
$Hom_{\mathbb{R}}(\Lambda_{\mathbb{R}},\mathbb{R})/Hom_{\mathbb{Z}}
(\Lambda,\mathbb{Z})$ where the complex structure on 
$Hom_{\mathbb{R}}(\Lambda_{\mathbb{R}},\mathbb{R})$ is defined by the 
unique operator $J$ such that $\langle J\omega,Iv\rangle = \langle 
\omega,v\rangle$. The canonical homomorphism $\varphi_{L}:P\to 
\hat{P}$ is given by the $\mathbb{C}$-linear map 
$\varphi:\Lambda_{\mathbb{R}}\to 
Hom_{\mathbb{R}}(\Lambda_{\mathbb{R}},\mathbb{R})$, where 
$\varphi(v)=E(v,-)$. Let $\Lambda^{*}$ be the dual lattice of 
$\Lambda$, i.e.
\[
\Lambda^{*} = \{v\in V|E(v,\lambda)\in \mathbb{Z} \text{\; for\;} \forall 
\lambda\in \Lambda\}.
\]
Then $\varphi$ induces an isomorphism 
$V/\Lambda^{*}\overset{\sim}{\lto}\hat{P}$. The dual polarization 
$L_{\delta}$ on $\hat{P}$ is given by the unique Riemann form $\hat{E}$ such 
that $\varphi^{*}\hat{E}=d_{1}d_{p}E$. So, the pair 
$(\hat{P},L_{\delta})$ is isomorphic to $(V/\Lambda^{*},d_{1}d_{p}E)$. 

If $(\gamma_{1},\ldots,\gamma_{p},\gamma_{p+1},\ldots,\gamma_{2p})$ is 
a symplectic basis of $\Lambda$, \; 
$E(\gamma_{i},\gamma_{p+j})=\delta_{ij}d_{i}$, then 
\[
(-\frac{1}{d_{p}}\gamma_{2p},\ldots,-\frac{1}{d_{1}}\gamma_{p+1},
\frac{1}{d_{p}}\gamma_{p},\ldots,\frac{1}{d_{1}}\gamma_{1})
\]
 is a 
symplectic basis of $\Lambda^{*}$ with respect to $d_{1}d_{p}E$. So, 
the polarization type of $\hat{E}$ is 
\[
(d_1,\frac{d_1d_p}{d_{p-1}},\ldots,\frac{d_1d_p}{d_{p-i+1}},\ldots,d_{p}).
\]
\end{block}
Let $(A=V/\Lambda,\Theta)$ be a principally polarized Abelian variety. 
Let $\delta:A\to A$ be an endomorphism, which is symmetric with 
respect to the Rosati involution of $(A,\Theta)$ and satisfies the 
equation $(\delta-1)(\delta-q+1)=0$ for some $q\geq 2$. Let 
$P=(1-\delta)A$. Let $E$ be the Riemann form of $\Theta$ and let 
$E_{P}$ be the Riemann form of the restricted polarization 
$\Theta|_{P}$. The next proposition is a generalization to  arbitrary polarizations
 of some material well known in the case of principal polarizations 
(see \cite{BM}, Section~7 and \cite{K}, Proposition~2.4).
\begin{pro}\label{s0.7}
Let $(Z=V_{Z}/\Lambda_{Z},\theta)$ be a polarized Abelian variety with 
polarization 
type $(d_1,\ldots,d_{p})$ and let $E_{\theta}$ be its Riemann form. 
Let $f:A\to Z$ be a surjective homomorphism which satisfies
\begin{equation}\label{es0.7a}
E_{\theta}(\tilde{f}(\alpha),\tilde{f}(\beta)) = 
E((1-\tilde{\delta})(\alpha),\beta)
\end{equation}
for $\forall \alpha,\beta \in \Lambda$. 

Then 
$\mu=\varphi_{\Theta}^{-1}\circ \hat{f}:\hat{Z}\to A$ transforms 
$\hat{Z}$ onto $P$, $f|_{P}:P\to Z$ and $\mu:\hat{Z}\to P$ are 
isogenies, and 
$$\mu^{*}E_{P}=\frac{q}{d_1d_p}\hat{E_{\theta}}.
$$ 
Furtermore $\mu: \hat{Z}\to P$ is an isomorphism if and only if 
$\tilde{f}:\Lambda \to \Lambda_{Z}$ is surjective.
\end{pro}
\begin{proof}
One has $P=V^{-}/\Lambda \cap V^{-}$, where $V^{-}$ is the eigenspace 
of $\tilde{\delta}:V\to V$ with eigenvalue $1-q$. Let $\phi =\tilde{f}:V\to 
V_{Z}$. Define ${^{t}\phi}:V_{Z}\to V$ by 
\[
E({^{t}\phi}(x),w) = E_{Z}(x,\phi(w))
\]
for all $x \in V_Z$ and $w \in V$.
This is a $\mathbb{C}$-linear map which induces 
$\varphi_{\Theta}^{-1}\circ \hat{f}\circ \varphi_{\theta}:Z\to 
A$.

\medskip \noindent
{\sc Claim}. The following properties hold:
\renewcommand{\theenumi}{\alph{enumi}}
\begin{enumerate}
\item
${^{t}\phi}:V\to V_{Z}$ is injective,
\item 
${^{t}\phi}\circ \phi = 1-\tilde{\delta}$,
\item
$\phi \circ \tilde{\delta} = (1-q)\phi$,
\item
$\phi \circ {^{t}\phi} = q\cdot id$.
\end{enumerate}

\medskip \noindent
Indeed, ${^{t}\phi}(x)=0$ iff $E({^{t}\phi}(x),w)=0$ for $\forall w\in 
V$. Since $E({^{t}\phi}(x),w)=E_{Z}(x,\phi(w))$ and $\phi:V\to V_{Z}$ 
is epimorphic by hypothesis, one obtains $x=0$. One has
\[
E({^{t}\phi}\circ \phi(v),w) = E_{\theta}(\phi(v),\phi(w)) = 
E((1-\tilde{\delta})v,w)
\]
for $\forall v,w\in V$. Therefore ${^{t}\phi}\circ \phi=1-\tilde{\delta}$. In 
order to prove (c) it suffices to verify that ${^{t}\phi}\circ 
\phi\circ \tilde{\delta} = (1-q){^{t}\phi}\circ \phi$. This follows from (b) 
and the equation $(\delta-1)(\delta+q-1)=0$. In order to prove (d) it 
suffices to verify that $\phi \circ {^{t}\phi}\circ \phi = q\phi$. 
This follows from (b) and (c). The claim is proved.

\medskip \noindent
Using (b) and (a) one has that 
${^{t}\phi}:V_{Z}\to V^{-}$ is an isomorphism, so $\mu:\hat{Z}\to P$ 
is an isogeny. Furthermore $Ker\: \phi = Ker( {^{t}\phi}\circ \phi) = 
Ker(1-\tilde{\delta}) = (\tilde{\delta}+q-1)V$. Therefore $f|_{P}:P\to Z$ is an 
isogeny. One has 
\[
E({^{t}\phi}(x),{^{t}\phi}(y)) = E_{\theta}(x,\phi\circ{^{t}\phi}(y)) 
= q\: E_{\theta}(x,y)
\]
Since ${^{t}\phi}$ induces 
$\varphi_{\Theta}^{-1}\circ \hat{f}\circ \varphi_{\theta} = \mu\circ 
\varphi_{\theta}$ we conclude that 
$\varphi_{\theta}^{*}(\mu^{*}E_{P})=q\: E_{\theta}$. This implies 
$\mu^{*}E_{P}=\frac{q}{d_1d_p}\hat{E_{\theta}}$ (cf. \S\ref{s0.5}).
The last statement follows from the isomorphism (cf. \cite{BL2}, 
Proposition~2.4.3)
\[
Ker(\hat{f}:\hat{Z}\to \hat{A}) \cong 
Hom(\Lambda_{Z}/\tilde{f}(\Lambda),\mathbb{C}^{*}).
\]
\end{proof}
\begin{block}\label{s10.24}
Recall that if $\tilde{Y}\to Y$ 
is a double covering of smooth, projective, irreducible curves then 
$\dim P(\tilde{Y},Y)=g(\tilde{Y})-g(Y)$ and the restriction of the 
canonical polarization $\Theta_{J\tilde{Y}}$ on $P(\tilde{Y},Y)$ has 
type $(2,\ldots,2)$ if $\tilde{Y}\to Y$ is unramified and type 
$(1,\ldots,1,2,\ldots,2)$ if $\tilde{Y}\to Y$ is ramified, where 2 
appears $g(Y)$ times and 1 appears $g(\tilde{Y})-2g(Y)$ times (cf. 
\cite{Fay}, \cite{Mum}).
Suppose 
$p:C\overset{\pi}{\lto}C'\overset{g}{\lto}\mathbb{P}^1$ is a simply ramified 
$B_{n}$-covering with irreducible $C$. Let $\mathfrak{D}=\mathfrak{D}_{s}\cup 
\mathfrak{D}_{\ell}$ be the splitting of the discriminant locus as in 
\S\ref{s0.1}. Let $P'=P(C,C')$ and let $E_{P'}$ be the Riemann form of the 
restriction $\Theta_{JC}|_{P'}$. 
 Calculating the genera of $C$ and $C'$ by the 
Hurwitz formula one obtains that the polarization type 
$(d_1,\ldots,d_{p})$ of $\Theta_{JC}|_{P'}$ is $(2,\ldots,2)$ if 
$|\mathfrak{D}_{s}|=0$ or $|\mathfrak{D}_{s}|=2$ and 
$(1,\ldots,1,2,\ldots,2)$ where 1 appears 
$\frac{1}{2}|\mathfrak{D}_{s}|-1$ times and 2 appears 
$\frac{1}{2}|\mathfrak{D}_{\ell}|+1-n$ times.
\end{block}
\begin{block}\label{s10.24a}
Let $\hat{P}'$ and $\hat{E}_{P'}$ be the dual Abelian 
variety and the dual polarization. Consider the covering $f:X\to 
\mathbb{P}^1$ of degree $2^{n}$ associated with the spinor weight 
$\omega_{n}$ as in \S\ref{s0.3}. Let $P=P(X,\delta)$ be the 
Prym-Tyurin variety. Here $A=JX,\: P=(1-\delta)JX$ and $q=2^{n-1}$. 
The correspondence $S_{0}\in Div(X\times C)$ constructed in 
\S\ref{s0.21} induces a homomorphism $s_{0}:JX\to P'=P(C,C')$. 

We claim the hypothesis of Proposition~\ref{s0.7} is satisfied for 
$s_{0}:JX\to P'$, where $(A,\Theta)=(JX,\Theta_{JX})$, 
$(Z,\theta)=(P',\Theta_{JC}|_{P'})$. Abusing notation we will write $s_0$ instead of 
$\tilde{s}_0$, or instead of the induced homomorphism on homology $(s_0)_*$, when 
things are clear from the context.

By Lemma~\ref{s0.23} $s_{0}:JX\to P'$ is 
surjective. Furthermore a standard fact is  
that the homomorphisms $s_{0}:JX\to JC$ and ${^{t}s}_{0}:JC\to JX$ 
induced by $S_{0}\in Div(X\times C)$ and its transpose ${^{t}S}_{0}\in 
Div(C\times X)$ satisfy the relation 
\[
(\alpha,s_{0}(\beta))_{C} = ({^{t}s}_{0}(\alpha),\beta)_{X}
\]
for $\forall \alpha \in H_{1}(C,\mathbb{Z})$ and $\forall \beta \in 
H_{1}(X,\mathbb{Z})$. Using Lemma~\ref{s0.23} one obtains the 
following relation of Riemann forms:
\[
E_{P'}(s_{0}(\alpha),s_{0}(\beta)) = E_{JX}((1-\tilde{\delta})\alpha,\beta)
\]
for $\forall \alpha,\beta \in H_{1}(X,\mathbb{Z})$. Applying 
Proposition~\ref{s0.7} we obtain the following result.
\end{block}
\begin{thm}\label{s10.25}
Let $p:C\overset{\pi}{\lto}C'\overset{g}{\lto}\mathbb{P}^1$ be a 
$B_n$-covering. 
Let $f:X\to \mathbb{P}^1$, $P=P(X,\delta)$ and $s_{0}:JX\to P'=P(C,C')$ be 
as above. Let $E_{P}, E_{P'}$ be the Riemann forms of $\Theta_{JX}|_P$, $\Theta_{JC}|_{P'}$ respectively
and let $(d_1,\ldots,d_p)$ be the polarization type of $\Theta_{JC}|_{P'}$. Then the homomorphism 
$\varphi^{-1}_{\Theta}\circ \hat{s}_{0}:\hat{P}'\to JX$ yields an 
isogeny $\mu: \hat{P}'\to P$ such that
\[
\mu^{*}E_{P}=\frac{2^{n-1}}{d_{1}d_{p}}\hat{E}_{P'}.
\]
\end{thm}
\begin{conjecture}\label{s10.25a}
Assume $p:C\overset{\pi}{\lto}C'\overset{g}{\lto}\mathbb{P}^1$ is a simply 
ramified $B_n$-covering and $C$ is irreducible. Let $\hat{P}'$ be the dual 
of the Prym variety $P(C,C')$. Then the homomorphism 
$\mu: \hat{P}'\to P(X,\delta)$ is an isomorphism. 

Equivalently, the 
polarization type of $\Theta_{JX}|_{P(X,\delta)}$ is 
$(2^{n-2},\ldots,2^{n-2})$ if $|\mathfrak{D}_{s}|=0$ or 
$|\mathfrak{D}_{s}|=2$ and 
$(2^{n-2},\ldots,2^{n-2},2^{n-1},\ldots,2^{n-1})$ where $2^{n-2}$ 
appears $\frac{1}{2}|\mathfrak{D}_{\ell}|+1-n$                         
times and $2^{n-1}$ appears $\frac{1}{2}|\mathfrak{D}_{s}|-1$ times, 
if $|\mathfrak{D}_{s}|>2$.
\end{conjecture}
The case $n=2$ is known to be true and is due to Mumford \cite{Mum}
(the case $|\mathfrak{D}_{s}|=0$), Dalalyan \cite{Da} (the case 
$|\mathfrak{D}_{s}|=2$) 
 and 
Pantazis \cite{Pa} (the case $|\mathfrak{D}_{s}|>2$) . We include a simple 
proof in 
Proposition~\ref{s1.30}. We give a proof in the case $n=3, 
|\mathfrak{D}_{s}|>0$  in Theorem~\ref{s1.35} below. 
The case of \'{e}tale $\pi: C\to C'$ is verified in the next proposition. 
Conjecture~\ref{s10.25a} remains open for simple $B_n$-coverings of $\mathbb{P}^1$ with $n\geq 4$ (cf. Remark~\ref{ptv1.2}).
\begin{pro}\label{s10.26}
Let $p:C\overset{\pi}{\lto}C'\overset{g}{\lto}\mathbb{P}^1$ be a 
simply ramified $B_n$-covering with irreducible $C$. Assume $\pi:C\to 
C'$ is \'{e}tale. Let $f:X\to \mathbb{P}^1$ be the covering of degree 
$2^{n}$ associated with the spinor weight. Then 
$\Theta_{JX}|_{P(X,\delta)}$ has polarization type 
$(2^{n-2},\ldots,2^{n-2})$.
\end{pro}
\begin{proof}
By Corollary~\ref{s0.57} the monodromy group of $p:C\to \mathbb{P}^1$ 
is $W(D_n)$. Hence $X=X_{0}\sqcup X_{1}$, where $X_{0}\to 
\mathbb{P}^1$ and $X_{1}\to \mathbb{P}^1$ are coverings of degree 
$2^{n-1}$ which correspond to the two semispinor weights of the root 
system of type $D_n$ (cf. \S\ref{s2.6}). 
One has $JX=JX_{0}\times JX_{1}$. Let $P=P(X,\delta)$.
The restriction of $s_{0}:JX_{0}\times JX_{1}\to P(C,C')$ on $JX_{0}$ and 
$JX_{1}$ is studied in Section~8.6 of \cite{K}, where correspondences on $X_i$
of exponent $2^{n-3}$ and associated endomorphisms $\delta_i:JX_i\to JX_i$ are
defined. Let $P'=P(C,C')$.
>From the polarized 
isomorphism of $P'$ with the Prym-Tyurin varieties 
$P(X_{0},\delta_{0})$ and $P(X_{1},\delta_{1})$ proved there (see 
\cite{K}, Proposition~8.6.14 and Theorem~6.4(iii)), it follows that 
$(s_{0}|_{X_{i}})_{*}:H_{1}(JX_{i},\mathbb{Z})\to 
H_{1}(P',\mathbb{Z})$ is surjective for $i=0$ or $i=1$. Applying 
Theorem~\ref{s10.25} and
Proposition~\ref{s0.7}   we conclude that 
$\mu: \hat{P}'\to P$ is an isomorphism and 
$\mu^{*}E_{P}=\frac{2^{n-1}}{4}\hat{E}_{P'}$. Therefore $E_{P}$ has 
type $(2^{n-2},\ldots,2^{n-2})$. 
\end{proof}

\section{Coverings with monodromy group contained in $W(R)$, with 
$rank(R)\leq 
3$}\label{s1}

\begin{block}\label{s1.2} {\bf Bigonal construction, $W=W(B_2)$} 
\cite{D2,Pa}.
Let 
$W=W(R)$ where $R$ is of type $B_{2}$. Consider the two fundamental 
weights $\omega_{1}$ and $\omega_{2}$. We are in the situation of 
\S\ref{s0.1} and \S\ref{s0.3}. Let 
$p:C\overset{\pi}{\lto}C'\overset{g}{\lto}Y$ be a simple 
$B_2$-covering with irreducible $C$. Equivalently, $\deg \pi = \deg g = 2$, 
$\mathfrak{D}_{s}=g(Discr(C\to C'))$ and 
$\mathfrak{D}_{\ell}=Discr(C'\to Y)$ are nonempty, and 
$\mathfrak{D}_{s}\cap \mathfrak{D}_{\ell}=\emptyset$. 

Let $f:X\to Y$ 
be the covering of degree 4 associated with $\omega_{2}$. The 
monodromy group of $p:C\to Y$ is $W(B_2)$, so $X$ is irreducible. Let 
$\sigma:X\to X$ be the involution defined in \S\ref{s2.1} and let 
$f:X\overset{\pi'}{\lto}X'\overset{g'}{\lto}Y$ be the corresponding decomposition of 
$f$. Calculating the action of the reflections on $W\omega_{2}$ one 
verifies that $Discr(X'\to Y)=\mathfrak{D}_{s}$ and $g'(Discr(X\to 
X'))=\mathfrak{D}_{\ell}$. Here the isogenous Prym-Tyurin varieties are the 
ordinary Prym varieties $P(C,C')$ and $P(X,X')$. 
\end{block}
\renewcommand{\theenumi}{\roman{enumi}}
\begin{pro}\label{s1.30}
$\text{}$
\begin{enumerate}
\item 
The polarization types of $\Theta_{JC}|_{P(C,C')}$ and 
$\Theta_{JX}|_{P(X,X')}$ are respectively
\begin{equation*}
\begin{split}
&(\underbrace{1,\ldots,1}_{\frac{1}{2}|\mathfrak{D}_{s}|-1},
\underbrace{2,\ldots,2}_{\frac{1}{2}|\mathfrak{D}_{\ell}|-1},
\underbrace{2,\ldots,2}_{2g(Y)}),\\
&(\underbrace{1,\ldots,1}_{\frac{1}{2}|\mathfrak{D}_{\ell}|-1},
\underbrace{2,\ldots,2}_{\frac{1}{2}|\mathfrak{D}_{s}|-1},
\underbrace{2,\ldots,2}_{2g(Y)}).
\end{split}
\end{equation*}
\item 
If $Y=\mathbb{P}^1$, then $\mu:\hat{P}(C,C')\to P(X,X')$ defined in 
Theorem~\ref{s10.25} is an isomorphism.
\end{enumerate}
\end{pro}
\begin{proof}
(i) This follows from the Hurwitz formula (cf. \S\ref{s10.24}).

(ii) Let $P=P(X,X')$ and $P'=P(C,C')$. Let us consider the isogeny 
$\mu: \hat{P}'\to P$ from Theorem~\ref{s10.25}. Let $E_{P}$ and 
$E_{P'}$ be the Riemann forms of $\Theta_{JX}|_{P}$ and 
$\Theta_{JC}|_{P'}$ respectively. Suppose first that 
$|\mathfrak{D}_{s}|>2$ and $|\mathfrak{D}_{\ell}|>2$. By (i) the 
polarization type of $E_{P}$ is the same as that of $\hat{E}_{P'}$. We 
have by Theorem~\ref{s10.25} that $\mu^{*}E_{P}=\hat{E}_{P'}$. 
Therefore $\mu:\hat{P}'\to P$ is a polarized isomorphism. 

Let $|\mathfrak{D}_{s}|=2$. Then $X'\cong \mathbb{P}^1, P(X,X')=JX$ and 
$E_{P}=E_{JX}$ is a principal polarization. Here $E_{P'}=2E_{\Xi}$ for a 
principal polarization $\Xi$ on $P'=P(C,C')$, $\hat{E}_{P'}$ has type 
$(2,\ldots,2)$ and $\psi$ of the decomposition
\[
\varphi_{2\Xi}:P'\overset{2\, id}{\lto}P'\overset{\psi}{\lto}\hat{P}'
\]
determines a polarized isomorphism between $(P',E_{\Xi})$ and 
$(\hat{P}',\frac{1}{2}\hat{E}_{P'})$. By Theorem~\ref{s10.25} one 
has $\mu^{*}E_{JX}=\frac{1}{2}\hat{E}_{P'}$. Therefore $P'\cong 
\hat{P}'\cong JX$. 

Let $|\mathfrak{D}_{\ell}|=2$. Then $C'\cong 
\mathbb{P}^1, P'=JC$ and $E_{P}=2E_{\Xi}$ for a principal polarization 
$\Xi$. Identifying $JC$ with its dual $\hat{P}'$ we have for 
$\mu:JC\to P$ the formula $\mu^{*}E_{P}=2E_{JC}$. Therefore $\mu$ is a 
polarized isomorphism of $JC$ with $(P(X,X'),\Xi)$.
\end{proof}
\begin{rem}\label{s1.32}
The isomorphism $\hat{P}(C,C')\cong P(X,X')$ is due to Pantasis 
\cite{Pa} and the isomorphism $P(C,C')\cong JX$, when 
$|\mathfrak{D}_{s}|=2$, is due to Dalalyan \cite{Da}. Our proof, based 
on Theorem~\ref{s10.25}, is new and simpler than the original 
ones.
\end{rem}

\begin{block}\label{s1.5} {\bf Recillas' construction, $W=W(A_3)=W(D_3)$}.
Let $p:C\to Y$  be a simply ramified covering of degree 4 with 
irreducible $C$. Suppose it cannot be 
decomposed through an \'{e}tale covering $C'\to Y$ of degree 2. Then 
the covering is primitive, so its monodromy group is $S_4$ 
(cf. \S\ref{s0.50}).  Furthermore
$p^{*}:JY\to JC$ is injective, $A=Ker(Nm_{p})$ is connected,
 and the 
polarization type of $\Theta_{JC}|_{A}$ is 
$(1,\ldots,1,4,\ldots,4)$ where $4$ appear $g(Y)$ times (cf. \cite{BL2}, 
Corollary 12.1.5). In the 
set-up of the introduction we have $W=S_4=W(R)$, where $R$ is of type 
$A_3$ and $\lambda =\omega_{1}$. Replacing $\lambda$ by the second 
fundamental weight $\lambda'=\omega_{2}$ one obtains a covering  
$f:X\overset{\pi}{\lto}X'\overset{g}{\lto}Y$, where $X$ is irreducible, $\pi$ 
is 
\'{e}tale of degree $2$ and $g$ is a simply ramified covering of degree 3. 
Constructing an appropriate correspondence (cf. \cite{K}, \S8.5.3) one easily verifies that 
$X$ is isomorphic to a curve in $C^{(2)}$, namely the closure of 
$\{x+y\in C^{(2)}|x\neq y, p(x)=p(y)\}$. This construction, when $Y\cong \mathbb{P}^1$,
is due to Recillas \cite{Re}.
Here we have
$i = 0, P(C,i)=A$ and $\Theta_{JC}|_{P(C,i)}$ has polarization type as above, while  
 $P(X,\delta)$ is the ordinary Prym variety 
$P(X,X')$, so $\Theta_{JX}|_{P(X,X')}$ has 
polarization 
type  $(2,\ldots,2)$. Summarizing one obtains the following statement, due to Recillas when $Y\cong \mathbb{P}^1$ 
\cite{Re}.
\end{block}
\begin{pro}\label{ptv1.3}
Let $p:C\to Y$  be a simply ramified covering of degree 4 with irreducible $C$, which cannot be decomposed through an 
\'{e}tale covering of degree 2. Let $f:X\overset{\pi}{\lto}X'\overset{g}{\lto}Y$ be the associated
covering of degree 6 obtained by the Recillas construction. Then $A=Ker(Nm_p:JC\to JY)$ and $P(X,X')$ are
isogenous Abelian varieties and the polarization types of $\Theta_{JC}|_{A}$ and $\Theta_{JX}|_{P(X,X')}$
are respectively
\begin{equation*}
\begin{split}
&(1,\ldots,1,
\underbrace{4,\ldots,4}_{g(Y)}),\\
&(2,\ldots,2,2\ldots,2)
\end{split}
\end{equation*}
If $Y\cong \mathbb{P}^1$, then $J(C)$ and $P(X,X')$ are isomorphic as principally polarized Abelian varieties.
\end{pro} 
\begin{block}\label{ptv1.4}
Replacing $\lambda=\omega_{1}$ by $\lambda'=\omega_{3}$ one obtains a 
covering $f:X\to Y$ which is equivalent to $p:C\to Y$, since $-id$ is 
$W(A_{3})$-equivariant and transforms $\omega_{1}$ to $\omega_{3}$. 
Hence there is a polarized isomorphism between $(A,\Theta_{JC}|_{A})$ 
and $(A',\Theta_{JX}|_{A'})$, where $A'=Ker(Nm_{f}:JX\to JY)$.

Inverting Recillas' construction one covers the case $W=W(D_3), 
\lambda=\omega_{1}, \lambda'=\omega_{2}$ or  
$\lambda'=\omega_{3}$
where $\omega_{1}, 
\omega_{2}, \omega_{3}$ are the fundamental weights of a root system of 
type $D_3$. 
\end{block}

\begin{block}\label{s1.5a}
The next case we consider is $W=W(R)$, where $R$ is of type $B_{3}$, 
$\lambda=\omega_{1}$ and $\lambda' = \omega_{3}$
 (cf. Section~\ref{s10}). 
Let $p:C\overset{\pi}{\lto}C'\overset{p'}{\lto}Y$
be a simple $B_3$-covering. Since $\deg p'=3$, so $p':C'\to Y$ is 
primitive, the simplicity of the covering $p$ is  equivalent to the condition that $p:C\to 
Y$ is a simply ramified  $B_3$-covering and both $\pi:C\to C'$ and $p':C'\to Y$ are ramified. 
Let $f:X\to Y$ be the covering of degree 8 associated with 
$\omega_{3}$. The curve $X$ is irreducible, since the monodromy group 
of $p:C\to Y$ is $W(B_3)$ by Corollary~\ref{s0.56}. Let 
$f:X\overset{g}{\lto}\tilde{Y}\overset{h}{\lto}Y$ be the decomposition 
defined in \S\ref{s2.6}.
\end{block} 
\begin{lem}\label{s1.34}
The following equalities hold
\renewcommand{\theenumi}{\roman{enumi}}
\begin{enumerate}
\item 
$P(X,\delta) = Ker(Nm_{g}:P(X,X')\to P(\tilde{Y},Y))^{0}$
\item 
$(\delta+3)P(X,X')=g^{*}P(\tilde{Y},Y)$.
\end{enumerate}
\end{lem}
\begin{proof}
In Proposition~\ref{prop2.7} we proved inclusions, so it only remains 
to compare the dimensions. We have $\dim P(X,\delta)=\dim P(C,C')$ 
since these are isogenous Abelian varieties by Proposition~\ref{s0.16}. Using \eqref{es2.1a}, 
\eqref{es2.1c} and \eqref{es2.6a} for $n=3$ we conclude that 
$\dim P(X,\delta)=\dim P(X,X')-P(\tilde{Y},Y)$, so (i) holds. 

One has $P(X,\delta)=(1-\delta)P(X,X')$ by Proposition~\ref{prop2.5} (ii). 
Furthermore $P(X,X')=(1-\delta)P(X,X')+(\delta+3)P(X,X')$, and the two 
Abelian varieties on the right intersect in a finite number of 4-torsion points. 
Therefore $\dim (\delta+3)P(X,X')=\dim P(\tilde{Y},Y)$. Hence (ii) 
holds as well.
\end{proof}
\begin{thm}\label{s1.35}
Let $p:C\to C'\to Y$ be a simple $B_3$-covering. Let $f:X\to Y$ be the 
covering of degree 8 associated with the spinor weight and let 
$P(X,\delta)$ be the Prym-Tyurin variety. Then the polarization types 
of $\Theta_{JC}|_{P(C,C')}$ and
 $\Theta_{JX}|_{P(X,\delta)}$ are respectively
\begin{align}
&(\underbrace{1,\ldots,1}_{\frac{1}{2}|\mathfrak{D}_{s}|-1},
\underbrace{2,\ldots,2}_{\frac{1}{2}|\mathfrak{D}_{\ell}|+2g(Y)-2},
\underbrace{2,\ldots,2}_{g(Y)})\label{es1.35a}\\
&(\underbrace{2,\ldots,2}_{\frac{1}{2}|\mathfrak{D}_{\ell}|+2g(Y)-2},
\underbrace{4,\ldots,4}_{\frac{1}{2}|\mathfrak{D}_{s}|-1},
\underbrace{8,\ldots,8}_{g(Y)})\label{es1.35}
\end{align}
If $Y\cong \mathbb{P}^1$, then Conjecture~\ref{s10.25a} holds: 
$P(X,\delta)$ is isomorphic to the dual of $P(C,C')$.
\end{thm}
\begin{proof}
By \S\ref{s2.1} one has 
$g(C')=\frac{1}{2}|\mathfrak{D}_{\ell}|+3g(Y)-2$, so \eqref{es1.35a}
follows from the well known fact about ordinary Prym varieties (cf. 
\S\ref{s10.24}).
By \S\ref{s2.6} the covering $g:X\to \tilde{Y}$ of degree 4  is simply 
branched in $h^{-1}(\mathfrak{D}_{\ell})$ and has monodromy group 
$W(D_3)\cong S_{4}$. Therefore $g$ cannot be decomposed through an 
\'{e}tale covering of $\tilde{Y}$, so $g^{*}:J\tilde{Y}\to JX$ is 
injective. We have a commutative diagram 
\begin{equation}\label{es1.36}
\begin{CD}
JX           @<g^*<<         J\tilde{Y} \\
@AAA                           @AAA      \\
P(X,\sigma)  @<g^*<<         P(\tilde{Y},Y)
\end{CD}
\end{equation}
The restriction $\Theta_{J\tilde{Y}}|_{P(\tilde{Y},Y)}$ has type 
$(1,\ldots,1,2,\ldots,2)$ where 1 appears 
$\frac{1}{2}|\mathfrak{D}_{s}|-1$ times and 2 appears $g(Y)$ times. 
The pull-back of $\Theta_{JX}$ by $g^{*}$ is $4\Theta_{J\tilde{Y}}$. 
The involution $\sigma:X\to X$ is fixed point free, so 
$\Theta_{JX}|_{P(X,X')}=2\Xi$ for a principal polarization $\Xi$. 
Using \eqref{es1.36}, the fact that $g^{*}$ is injective and the 
equality $(\delta+3)P(X,X')=g^{*}P(\tilde{Y},Y)$ of 
Lemma~\ref{s1.34} we conclude that $\Xi|_{(\delta+3)P(X,X')}$ has 
polarization type $(2,\ldots,2,4,\ldots,4)$ where 2 appears 
$\frac{1}{2}|\mathfrak{D}_{s}|-1$ times and 4 appears $g(Y)$ times.
 
Recall that a polarization $L$ on an Abelian variety $B$ determines 
$\varphi_{L}:B\to \hat{B}$ and $K(L):=Ker(\varphi_{L})$.  If 
$(A,\Theta)$ is a principally polarized Abelian variety with a 
Prym-Tyurin endomorphism $\delta$ and if $B=(\delta+q-1)A, P=(1-\delta)A$, then 
$B$ and $P$ are a pair of complementary Abelian subvarieties of $A$. Hence the
finite groups $K(\Theta|_B)$ and $K(\Theta|_P)$ are isomorphic according to 
\cite{BL2}, Corollary~12.1.5.
Applying this to 
$(A,\Theta)=(P(X,X'),\Xi)$ we conclude that the polarization of 
$\Xi|_{P(X,\delta)}$ is $(1,\ldots,1,2,\ldots,2,4,\ldots,4)$ where the 
lengths of the sequences $(2,\ldots,2)$, $(4,\ldots,4)$ and 
$(1,\ldots,1)$ are respectively $\frac{1}{2}|\mathfrak{D}_{s}|-1, 
g(Y)$ and $\dim P(X,\delta)-\dim(\delta+3)P(X,X')$. The latter 
number equals 
$\dim P(X,\delta)-P(\tilde{Y},Y)=\frac{1}{2}|\mathfrak{D}_{\ell}|+2g(Y)
-2$. Since $\Theta_{JX}|_{P(X,X')}=2\Xi$ we obtain, restricting on 
$P(X,\delta)$ the polarization type of \eqref{es1.35}

Let $Y=\mathbb{P}^1$. Let $P=P(X,\delta)$ and $P'=P(C,C')$. One has 
$|\mathfrak{D}_{\ell}|\geq 4$ since $p':C'\to \mathbb{P}^1$ is a 
triple covering, simply ramified in $\mathfrak{D}_{\ell}$. By 
hypothesis $\pi:C\to C'$ is ramified, so $|\mathfrak{D}_{s}|\geq 2$.
 
Suppose first $|\mathfrak{D}_{s}|>2$ and $|\mathfrak{D}_{\ell}|>4$. By 
Theorem~\ref{s10.25} we have $\mu^{*}E_{P}=2\hat{E}_{P'}$. By 
\eqref{es1.35a} and \eqref{es1.35}, applied with $g(Y)=0$, the 
polarization types of $E_{P}$ and $2\hat{E}_{P'}$ are the same. 
Therefore $\mu:\hat{P}'\to P$ is an isomorphism. 

Let $|\mathfrak{D}_{s}|=2$. Then $\tilde{Y}\cong \mathbb{P}^1$, 
$P=P(X,X')$, and $E_{P}=2E_{\Xi}$ for some principal polarization $\Xi$ 
on $P$. Since $\pi:C\to C'$ is ramified in two points we have 
$E_{P'}=2E_{\Xi'}$ for some principal polarization $\Xi'$ of 
$P'=P(C,C')$. By Theorem~\ref{s10.25} we have 
$\mu^{*}E_{P}=\hat{E}_{P'}$. Therefore $\mu^{*}E_{\Xi}=E_{\Xi'}$ and 
we have  polarized isomorphisms $(P(C,C'),E_{\Xi'})\cong 
(\hat{P}(C,C'),\hat{E}_{\Xi'})\cong (P(X,X'),E_{\Xi})$. 

Let $|\mathfrak{D}_{\ell}|=4$. Then $C'\cong \mathbb{P}^1$ and 
$P'=P(C,C')=JC$. Here $P\subset P(X,X')$ and $E_{P}=4E_{\Xi}$ for a 
principal polarization $\Xi$, according to \eqref{es1.35}. By 
Theorem~\ref{s10.25} we have $\mu^{*}E_{P}=4\hat{E}_{JC}$. 
Identifying $JC$ with its dual we obtain a polarized isomorphism 
$(JC,\Theta_{JC})\cong (P(X,\delta),\Xi)$. 
\end{proof}
The polarized isomorphism $P(C,C')\cong P(X,X')$ 
is a particular case of the tetragonal construction of Donagi. 
The limit case $|\mathfrak{D}_{\ell}|=4$ yields a new calss of 
principally polarized Prym-Tyurin varieties. 
\begin{cor}\label{s1.39}
Let $\pi:C\to \mathbb{P}^1$ be a hypereliptic curve. Let 
$p':C'=\mathbb{P}^1\to \mathbb{P}^1$ be a simple triple covering, so 
that $p:C\overset{\pi}{\lto}C'\overset{p'}{\lto}\mathbb{P}^1$ is a 
simply ramified $B_3$-covering (cf. \S\ref{s0.1}). Let $f:X\to 
\mathbb{P}^1$  be the associated covering of degree 8 (cf. 
\S\ref{s0.3}), and let $\delta:JX\to JX$ be the associated 
Prym-Tyurin endomorphism, satisfying $(\delta-1)(\delta+3)=0$. Then the 
Prym-Tyurin variety $P(X,\delta)=(1-\delta)JX$ has a principal 
polarization $\Xi$ with the property that 
$\Theta_{JX}|_{P(X,\delta)}\equiv 4\Xi$. Furthermore $(P(X,\delta),\Xi)\cong 
(JC,\Theta_{JC})$. 
\end{cor}
\begin{rem}\label{s1.39a}
We notice that in this example the correspondence $D\in Div(X\times 
X)$ has 8 fixed points, two ones in each of the fibers $f^{-1}(y)$ 
where $y$ belongs to the branch locus of $p':\mathbb{P}^1\to 
\mathbb{P}^1$. The equality $\Theta_{JX}|_{P(X,\delta)}\equiv 4\Xi$
cannot be deduced from Ortega's criterion 
\cite{Or}, since this criterion requires the existence of 4 fixed 
points $p_{1},\ldots,p_{4}$ of $D$ which satisfy 
$p_{1},\ldots,p_{k}\in D(p_{k})$ and $p_{k}\notin D(p_{k})-p_{k}$ for 
$k=1,\ldots,4$. Hence these 4 points have to belong to a single fiber 
of $f:X\to \mathbb{P}^1$, which is not the case in 
Corollary~\ref{s1.39}. 
\end{rem}
\begin{block}\label{s1.40}
We studied in \S{\S}\ref{s1.5a}--\ref{s1.39a} simple 
$B_3$-coverings. Let us consider another class of simply 
ramified $B_3$-coverings. Namely, let 
$p:C\overset{\pi}{\lto}C'\overset{p'}{\lto}Y$ be with unramified $\pi$  
 and let $p$ have monodromy group $W(D_3)$ (cf. 
Corollary~\ref{s0.56}). 
 Since $W(D_3)$ has two orbits when acting on 
$W\omega_{3}$, namely $\{\lambda_{A}\}_{A\; \text{even}}$ and 
$\{\lambda_{A}\}_{A\; \text{odd}}$, the covering $f:X\to Y$ splits into 
disjoint union 
$$
f:X_{1}\sqcup X_{2}\to Y.
$$
Let $f_{i}=f|_{X_{i}}$. The 
canonical involution $\sigma:X\to X$ associated with 
$-id:W\omega_{3}\to W\omega_{3}$ transforms isomorphically $X_{1}$ and 
$X_{2}$ into each other. If $x\in X_{i}$ we have 
\begin{equation}\label{es1.8}
D(x)\ =\ f^{*}_{i}(f_{i}(x)) - x +2\sigma(x)
\end{equation}
Let $\varphi :Z\to Y$ be a fixed covering of degree 4 equivalent to 
any of $f_{i}:X_{i}\to Y,\quad i=1,2$. The covering $\varphi:Z\to Y$ 
is simply ramified with monodromy group $W(D_3)=S_{4}$. 
 We have $JX = JX_{1}\times 
JX_{2}\cong JZ\times JZ$. Let 
$$
B=Ker(Nm_{\varphi}:JZ\to JY).
$$ 
Identifying $A$ with the kernel of the homomorphism $JZ\times JZ\to JY$ given
by $(a,b)\mapsto \varphi(a)+\varphi(b)$ we see that 
under the map $\varphi \times \varphi$ the variety
$A$ maps onto the 
antidiagonal  
$$
\Delta^{-}(JY\times JY):=\{(\lambda,-\lambda)|\lambda \in JY\}.
$$ 
One obtains an exact sequence
\[
0 \lto B\times B \lto A \lto JY \lto 0
\]
\end{block}
\begin{lem}\label{s1.8}
The Prym-Tyurin variety $P(X,\delta)=(1-\delta)A$ equals the antidiagonal of $B\times 
B$, $P(X,\delta)=\Delta^{-}(B\times B)$.
\end{lem}
\begin{proof}
If $(\lambda,\mu)\in B\times B$, then \eqref{es1.8} yields 
\(
(1-\delta)(\lambda,\mu)\ =\ 2(\lambda-\mu,\mu-\lambda),
\)
so
\begin{equation}\label{es1.40}
(1-\delta)\Delta(B\times B) = 0, \quad (1-\delta)\Delta^{-}(B\times B)
= \Delta^{-}(B\times B)
\end{equation}
Let us apply the results of Section~\ref{s2} to the type of coverings 
we consider (see Remark~\ref{ptv1.5}). Here $\tilde{Y}=Y\sqcup Y$, the covering $g:X\to 
\tilde{Y}$ is $\varphi \sqcup \varphi : Z\sqcup Z\to Y\sqcup Y$, 
$P(X,X')=\Delta^{-}(JZ\times JZ)$, $P(\tilde{Y},Y)=\Delta^{-}(JY\times 
JY)$. One has the canonical decomposition $JZ=\varphi^{*}(JY)+B$, so 
\[
\Delta^{-}(JZ\times JZ) = 
(\varphi^{*}\times \varphi^{*})\Delta^{-}(JY\times JY) + \Delta^{-}(B\times B)
\]
By Proposition~\ref{prop2.7}~(ii) one has 
\[
(\varphi^{*}\times \varphi^{*})\Delta^{-}(JY\times JY)\subset (\delta+3)\Delta^{-}(JZ\times JZ)
\]
so $(1-\delta)$ annihilates this Abelian subvariety. By Proposition~\ref{prop2.5}~(ii) one has
$P(X,\delta)=(1-\delta)\Delta^{-}(JZ\times JZ)$, so using \eqref{es1.40} we obtain 
$P(X,\delta)=\Delta^{-}(B\times B)$.
\end{proof}
\begin{pro}\label{s1.9}
Let $p:C\overset{\pi}{\lto}C'\overset{p'}{\lto}Y$ be a simply ramified 
$B_3$-covering with unramified $\pi$ and monodromy group $W(D_3)$.
Let $f:X\to Y$ be the associated covering of degree 8. Then the 
polarization type of $\Theta_{JX}|_{P(X,\delta)}$ is 
$(2,\ldots,2,8,\ldots,8)$ where 8 appears $g(Y)$ times.
\end{pro}
\begin{proof}
Let $E_{B}$ be the Riemann form of $\Theta_{JX}|_{B}$. Since $\varphi 
:Z\to Y$ is primitive, the polarization type of $E_{B}$ is 
 $(1,\ldots,1,4,\ldots,4)$ where $4$ appears $g(Y)$ times (cf. 
\S\ref{s1.5}). By Lemma~\ref{s1.8}
 $P(X,\delta)=P\subset B\times B$ is the antidiagonal, so denoting by 
$\tilde{{}}$ the universal covering, we have 
$E_{P}=(p^{*}_{1}E_{B}+p^{*}_{2}E_{B})|_{_{\tilde{P}}}$ and for every 
$u=(x,-x)\in \tilde{B}\times \tilde{B}$ and every $v=
(y,-y)\in \tilde{B}\times \tilde{B}$ we have 
\begin{equation*}
E_{P}(u,v) = (p^{*}_{1}E_{B}+p^{*}_{2}E_{B})(x,-x;y,-y) = 2E_{B}(x,y)
\end{equation*}
Therefore the polarization type of $E_{P}$ is $(2,\ldots,2,8,\ldots,8)$. 
\end{proof}

\section{Rank(R) = 4}\label{s6}

\begin{block}\label{s4.1} {\bf The case $W = W(B_4)$.} Let $p: C \stackrel{g}{\ra} C' \stackrel{g}{\ra} Y$
be a simple $B_4$-covering and let $f: X \ra Y$ the covering of degree 16 associated with the spinor weight $\omega_4$.
The curve $X$ is irreducible, since the monodromy group of the covering $p$ is $W(B_4)$ by Corollary~\ref{s0.56}. The involution 
' on $C$ induces an involution $\sigma$ on the curve $X$ (see \S\ref{s2.1}). Denote $X' := X/\sigma$. 
We cannot say much about the induced polarizations on Prym-Tyurin varieties involved. However we can say something about the
relations between these varieties. 

For $i = 0, \ldots, 3$ define a correspondence $D_i$ of $X$ by
\[
D_i(x_{\emptyset}) = \sum_{|B|=i+1} x_B.
\]
So $D = \sum_{i=0}^3 iD_i$. 
\end{block}

\begin{pro} \label{prop4.2} If $\delta_i$ denotes the endomorphism of $JX$ 
associated to $D_i$, we have
\[
P(X,\delta) = (\delta_0 + 2)P(X,X').
\]
\end{pro} 
\begin{proof}
For any $x = x_{\emptyset} \in X$ we have
\[
(D -1)(x - \sigma(x)) = 
2(D_2 - D_0 )(x) + 4 \sigma(x) - 4x.
\]
Now $D_2(x) = D_0(\sigma(x))$ implies
$(D -1)(x - \sigma(x)) = 2(D_0 + 2) 
(\sigma(x) - x).$
According to Proposition \ref{prop2.5}(ii) this gives the assertion.
\end{proof}
According to Proposition \ref{prop2.5}(i), $P(X,\delta)$ is an abelian subvariety of 
the Prym variety $P(X,X')$. Hence we 
can speak about its complementary variety with respect to the polarization.
\begin{pro} \label{prop4.3}
The complementary abelian subvariety of $P(X,\delta)$ in $P(X,X')$ is given by
\[
(\delta + 7)P(X,X') = (\delta_0 - 2)P(X,X').
\]
\end{pro}
\begin{proof}
According to Propositions \ref{prop2.5}(ii) and \ref{s0.14a} the complementary 
abelian subvariety is given by 
$(\delta + 7)P(X,X')$. From the proof of Proposition \ref{prop4.2} we see that
$(D+7)(x - \sigma( x)) = 2(D_0 
-2)(\sigma(x) - x).$
This implies the last assertion.  
\end{proof}
Let $\rho:X \to X'$ denote the double covering associated to the involution 
$\sigma$. The next proposition describes the endomorphism
$(\delta_0 +2)(\delta_0 - 2)$ on $\rho^*JX'$.

\begin{pro} \label{prop4.4}
For any $x = x_{\emptyset} \in X$,
\[
(D_0+2)(D_0-2)(x + \sigma(x)) = 4 D_1(x).
\]
In particular $(\delta_0+2)(\delta_0 - 2)(\rho^*(JX')) = \delta_1(JX)$.
\end{pro}

\begin{proof}
For any $x = x_{\emptyset} \in X, \;\;(D_0 - 2)(x) = 
\sum_{|B|=1}x_B - 2x$ and hence
\[
(D_0 + 2)(D_0 - 2)(x)
= 4x_{\emptyset} + 2 \sum_{|B|=2}x_B - 2 \sum_{|B|=1}x_B + 2 \sum_{|B|=1}x_B 
-4x_{\emptyset}
\]
\hspace{3.9cm}  $= 2 D_1(x)$.

\vspace{0.3cm} \noindent
Similarly we get $(D_0 + 2)(D_0 - 2)(\sigma(x)) =2 
D_1(x)$. Adding both equations gives the assertion.
\end{proof}

\begin{block}\label{s4.2} {\bf The case $W = W(D_4)$.}
Let $p: C \stackrel{\pi}{\ra} C' \stackrel{p'}{\ra} Y$ denote a simply ramified $B_4$-covering with $\pi$ unramified.
Suppose moreover that $p$ has monodromy group $W(D_4)$ (cf. Corollary~\ref{s0.56}). 
We notice that this condition is automatically satisfied if $Y\cong \mathbb{P}^1$ by  Corollary~\ref{s0.57}.
 When $g(Y)\geq 1$ the monodromy group might be $W(B_n)$ or conjugated to $N_{W(B_n)}(G_1)$ (see 
Remark~\ref{ptv1.1}),
 cases which we do not consider in 
the present paper.
As in \S\ref{s1.40} the curve $X$, associated to the spinor weight $\omega_4$, 
is the disjoint union of two smooth irreducible curves $X = X_1 \sqcup X_2$.
For $i=1$ and $2$ the map $f_i := f|X_i: X_i \ra Y$
is an $8:1$ covering. Using the notation of above, we have for any $x = x_{\emptyset} \in X$,
$f_1^{-1}(f(x)) = \{x_B \}_{|B|\; even}$
and $f_2^{-1}(f(x)) = \{x_B \}_{|B| \;
odd}$. Hence the involution $\sigma : X \ra X$ induced by the involution ' on $C$ 
restricts to involutions $\sigma_i = \sigma|X_i : X_i \ra X_i$ for $i=1$ 
and 2. They induce factorizations of $f_i$, 
namely
$$
f_i: X_i \to X'_i \to Y.
$$
The map $X_i \to X'_i = X_i/\sigma_i$ is an \'etale double 
covering, since 
$\sigma$ and thus the involution $\sigma_i$ 
are fixed-point free. The map $X'_i \to Y$ is a 4:1 covering, simply ramified 
exactly 
over the $|D_l|$ ramification points of $p: C \ra Y$. This 
implies $g(X'_i) = \frac{|D_l|}{2} + 4g(Y) - 3 = g(C')$
and hence 
$$
\dim P(X_i,X'_i) = \frac{|D_l|}{2} + 4g(Y) - 4 = \dim P(C,C').
$$

\noindent
Consider the correspondence $D$ on $X$ defined in \S\ref{s0.3}.
\end{block}
%


\begin{lem} \label{lem4.6}
For any $y \in Y$ and $i=1$ and 2 we have
$$
D(f_i^*(y)) = 8 f^*(y) + f_i^*(y).
$$
\end{lem}

\begin{proof}
We prove this only for $i=1$, the proof being the same for $i=2$. 
Let $x = x_{\emptyset} = p_1 + \cdots + p_4 \in X$ with $f_i(x) = y$. Moreover 
let $\Sigma = \{1,2,3,4\}$ and all 
sets $B$ are subsets of $\Sigma$.
Then $f_1^*(y) = x + \sum_{|B| =2}x_B + x_{\Sigma}$ and hence
$$
\begin{array}{rl}
Df_1^*(y) = & D(x) + \sum_{|B|=2}D(x_B) + D(x_{\Sigma})\\
          =  &  \sum_{|B|=2}x_B + 2 \sum_{|B|=3}x_B + 3 x_{\Sigma}\\
            & + 6 x_{\emptyset} + 6x_{\Sigma} +7 \sum_{|B|=2}x_B + 2(3 
\sum_{|B|=1}x_B + \sum_{|B|=3}x_B)\\
            & + \sum_{|B|=2}x_B + 2 \sum_{|B|=1}x_B +3 x_{\emptyset}\\
          =  & 9x_{\emptyset} + 9x_{\Sigma} + 9 \sum_{|B|=2}x_B + 8 
\sum_{|B|=1}x_B + 8 \sum_{|B|=3}x_B\\
          =  & 8 f^*(y) + f_1^*(y).
\end{array}
$$
\end{proof}





Consider the abelian variety $A = ker (Nm_f: JX \ra JY)$ and for $i=1$ and 2 define
$B_i = ker (Nm_{f_i}: JX_i \ra JY).$ Then $B_1 \times B_2$
and $Q = \{(f_1^*(m), -f_2^*(m))\;|\; m \in JY \}$ are abelian subvarieties of $A$ and 
it is easy to see that $A = B_1 \times B_2 + Q$. 

\begin{pro} \label{prop4.8}
Denoting by $\delta$ also the restriction to $A$ of the endomorphism defined by $D$,
we have $(\delta - 1)Q=0$ and $Q \subset (\delta +7)A$ in concordance with $q = 8$. 
\end{pro}

\begin{proof}
For any divisor $c$ of degree 0 on $Y$ we have
$$
\begin{array}{rl}
(\delta -1)(f_1^*(c),-f_2^*(c)) = & (D-1)f_1^*(c) - (D-1)f_2^*(c)\\
= & 8f^*(c) + f_1^*(c) - f_1^*(c) - 8f^*(c) - f_2^*(c) + f_2^*(c) \\
= & 0.
\end{array}
$$
This proves the first assertion.
According to Lemma \ref{lem4.6}, $(D-1)(f_1^*(y)) = 8f^*(y)$. This implies
$(D+7)(f_1^*(y)) = 8f^*(y) + 8f_1^*(y)$
and similarly $(D+7)(-f_2^*(y)) = -8f^*(y) - 8f_2^*(y)$. 
Hence $(D+7)(f_1^*(y),-f_2^*(y)) = 8(f_1^*(y),-f_2^*(y))$
which implies the second assertion.
\end{proof}

\bibliographystyle{amsalpha}
\providecommand{\bysame}{\leavevmode\hbox to3em{\hrulefill}\thinspace}

\end{document}